\documentclass[11pt,a4paper]{article}

\usepackage{graphicx,amssymb,amsfonts,latexsym,amsmath,amsthm,times,subfigure}
\usepackage{authblk} 
\usepackage{epsfig}
\usepackage{epstopdf}
\usepackage{fancyhdr}

\usepackage{color}
\usepackage{xcolor}
\definecolor{darkred}{rgb}{0.6,0,0} 

\usepackage{lscape}
\usepackage[english]{babel}
\usepackage{authblk}
\usepackage{fancyhdr}
\usepackage{url}
\usepackage[utf8]{inputenc}
\usepackage{amsmath}
\usepackage{amsfonts}
\usepackage{amssymb}
\usepackage{tabularx}
\usepackage{float}
\usepackage{verbatim}
\usepackage{rotating}
\usepackage{graphics}
\usepackage{hyperref}
\usepackage{extarrows}
\usepackage{xcolor}
\usepackage{lscape}
\usepackage{multirow}
\usepackage{afterpage}
\usepackage{longtable}
\usepackage[T1]{fontenc}
\usepackage{calligra}
\usepackage{algorithm}
\usepackage{algpseudocode}
\usepackage{bm}
\usepackage{mathtools}
\usepackage[top=2.5cm,bottom=2.5cm,left=2.5cm,right=2.5cm]{geometry}
 \usepackage{hyperref}

\usepackage{url}  
\usepackage{tikz} 

\algrenewcommand\algorithmicindent{1.0em}

\allowdisplaybreaks

\newcommand{\Ealpha}{E_{\alpha}}
\newtheorem{thm}{Theorem}[section]
\newtheorem{lem}[thm]{Lemma}
\newtheorem{defn}{Definition}[section]
\newtheorem{rem}{Remark}

\DeclareMathOperator{\mem}{mem} 
\DeclareMathOperator{\tr}{tr} 
\DeclareMathOperator{\interior}{int} 
\DeclareMathOperator{\Real}{Re}

\begin{document}
\title{Structure-Preserving 
Schemes for a Fractional SVIR Epidemic Model with a Hybrid Mittag-Leffler-Caputo-Fabrizio Operator}

\author{Seham M. Al-Mekhlafi$^{1}$, Ahmed Boudaoui$^{2},$ Matthias Ehrhardt$^{3,}$\footnote{Corresponding Author, email: ehrhardt@uni-wuppertal.de}}
	\date{$^1$Department of Mathematics, Faculty of Education, Sana'a University, Sana'a, Yemen\\ 
		$^2$ Mathematics Modeling and Applications Laboratory,  University of Adrar, Adrar, Algeria\\
       $^3$ Applied and Computational Mathematics, University of Wuppertal, Gaußstrasse 20, Wuppertal, 42119, Germany\\
		E-mail: sih.almikhlafi@su.edu.ye$^1,$   ahmedboudaoui@univ-adrar.edu.dz$^2$,
        ehrhardt@uni-wuppertal.de$^{3,*}$}
	
	\maketitle

    \begin{tikzpicture}[remember picture,overlay]
\node[anchor=north east,inner sep=20pt] at (current page.north east)
	{\includegraphics[scale=0.2]{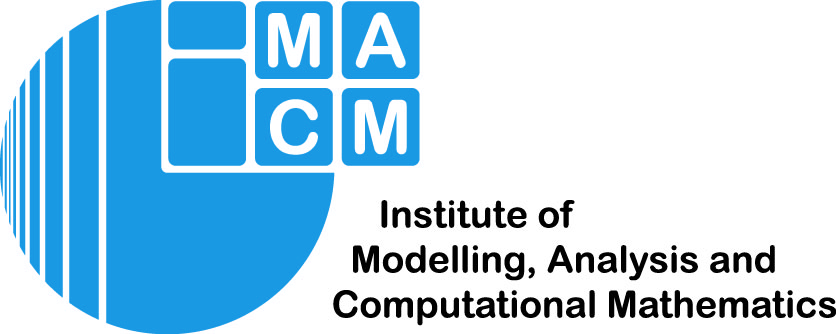}};
\end{tikzpicture}

\begin{abstract}
This paper proposes and analyzes a fractional-order SVIR epidemic model based on a hybrid Mittag-Leffler-Caputo-Fabrizio (MLCF) fractional operator with a nonsingular kernel. 
This model captures short- and long-term memory effects in epidemic transmission dynamics. 
The positivity and boundedness of the solutions are proven through an integrated formulation of the MLCF operator and a fractional Gronwall inequality. 
The basic reproduction number $\mathcal{R}_0$, equilibrium points, and their local and global stability properties are rigorously investigated through Jacobian analysis, logarithmic Lyapunov functionals, and a fractional LaSalle invariance principle.

To approximate the model, a $\theta$-weighted nonstandard finite difference (NSFD) method is developed. 
This method preserves the continuous system's key qualitative properties, including positivity and boundedness, and is unconditionally stable in the fully implicit case. 
Consistency and first-order convergence are also proven. 
Numerical experiments, together with sensitivity and bifurcation analyses, illustrate the impact of fractional memory parameters on epidemic evolution and demonstrate the effectiveness of the proposed approach.
\end{abstract}


\section{Introduction}
Mathematical modeling is increasingly being used as a primary tool to understand and control the spread of infectious diseases. 
One of the earliest classical compartmental models is the SIR scheme by Kermack and McKendrick \cite{Kermack1927}, which set the basis for contemporary epidemiological modeling.
Changes to the model, including factors such as vaccination and immunity loss, have produced more realistic versions, such as the SVIR model. This model is useful for studying diseases with immunization strategies \cite{Brauer2019, Hethcote2000}.
However, classical integer-order models struggle to represent the memory and hereditary effects that are essential to bioscience and epidemiological processes.
Recently, fractional calculus has been recognized as a suitable mathematical tool for modeling such phenomena due to its nonlocal character \cite{Diethelm2010, Podlubny1999}. 
Fractional-order epidemic models have demonstrated their ability to more accurately portray long-term dynamics and transmission mechanisms dependent on memory \cite{Chen2021, Yusuf2018}.

Traditional fractional derivatives, such as the Caputo operator, involve singular kernels that may hinder their applicability to certain real-life systems. 
Therefore, non-singular fractional operators have been developed, mainly the \textit{Caputo-Fabrizio} and \textit{Atangana-Baleanu derivatives}  \cite{AtanganaBaleanu2016, Caputo2015}, to overcome these limitations. 
By removing kernel singularities, these operators offer enhanced numerical stability and a more accurate physical interpretation. 
Additionally, the Mittag-Leffler function \cite{Gorenflo2014} has been instrumental in developing generalized fractional operators with superior modeling capabilities.
Recently, fractional operators, which are a hybrid of Mittag-Leffler and Caputo-Fabrizio kernels, have been proposed as a model for various memory effects. 
This new type of hybrid enables the simultaneous depiction of long- and short-term memory effects, which are characteristic of complex dynamical systems, such as epidemics \cite{Kavitha2025}.

A qualitative analysis of fractional epidemic models in terms of dynamical systems is crucial for determining their biological significance. 
Basic features such as positivity, boundedness, and the existence of invariant regions help make the model well-posed. 
The \textit{basic reproduction number} $\mathcal{R}_0$ is the primary threshold parameter that determines whether an infection persists or is eradicated \cite{Diekmann1990}.
Stability analysis of fractional systems requires modified versions of classical techniques and Lyapunov methods, as well as generalized stability criteria for fractional-order systems, cf.\ e.g.\ \cite{Li2010, Matignon1996}.

In computing, solving fractional differential equations remains a significant challenge due to their memory-intensive nature. 
Mickens's \textit{nonstandard finite difference} (NSFD) schemes have been shown to successfully maintain qualitative features such as positivity and boundedness \cite{Mickens2005}. 
Along with proper discretizations of fractional operators, these schemes serve as reliable instruments for the numerical study of epidemiological models. 
Recent research indicates that schemes that preserve the structure of fractional epidemic systems are crucial for accurately depicting their dynamics \cite{Shafqat2026}.

This paper presents and examines a fractional-order SVIR epidemic model driven by a hybrid \textit{Mittag-Leffler-Caputo-Fabrizio} (MLCF) operator. 
This model can capture short- and long-term memory effects within a unified, nonsingular framework. 
We characterize the disease-free and endemic equilibria and derive 
$\mathcal{R}_0$ using the next-generation matrix approach \cite{Van}. 
Local and global stability properties are rigorously investigated through Jacobian analysis, logarithmic Lyapunov functionals, and a fractional LaSalle invariance principle. 
Additionally, we propose a $\theta$-weighted 
NSFD method tailored to the MLCF operator.
This method combines forward differences, trapezoidal quadrature, and recursive memory updates to efficiently compute the model. 
The resulting scheme preserves the continuous model's essential qualitative properties, including positivity, boundedness, and unconditional stability in the fully implicit case. Finally, numerical simulations, sensitivity analyses, and bifurcation investigations demonstrate the influence of fractional memory parameters on epidemic dynamics and the effectiveness of the proposed approach for fractional epidemiological modeling.

The paper is organized as follows. Section \ref{sec:SVIRmodel} presents the proposed fractional SVIR model and the necessary preliminaries. Section \ref{S3} investigates the qualitative properties of the model, including positivity, boundedness, and stability analysis. In Section \ref{S4}, a $\theta$-weighted NSFD scheme is developed for the numerical approximation of the model. Section \ref{S5} is devoted to the analysis of the numerical method, including stability and convergence results. Sensitivity and bifurcation analyses are presented in Section \ref{S6}, while numerical simulations are provided in Section \ref{S7}. Finally, concluding remarks are given in Section \ref{S8}.
\section{Mathematical Formulation of the Fractional SVIR Model}\label{sec:SVIRmodel}
\subsection{Preliminaries}\label{subsec21}
First, let us revisit some fundamental definitions from the literature that are relevant to this work.
\begin{defn}[\cite{Hammouch2025}]\label{df1}
Let $0 < v < 1$. The Caputo fractional derivative for $u(t)$ is defined as 
\begin{equation}\label{def21}
    ^C D_t^v [u(t)] = \frac{1}{\Gamma(1 - v)} \int_0^{t} (t - \kappa)^{-v} \,u'(\kappa) \,d\kappa,
\end{equation}
where $t>0$, $u\in C[0,t]$, with $\Gamma(\cdot)$ denoting the Gamma function.
\end{defn}

 \begin{defn}[\cite{Caputo2015}]\label{df2}
 Let $0<v<1$. The Caputo-Fabrizio derivative is expressed as follows:
\begin{equation}\label{def22}
    ^{\rm CF}D_t^v [u(t)] = 
    \frac{M(v) \exp \Bigl[ -v \frac{t}{1 - v} \Bigr]}{1 - v} \int_0^{t} u'(\kappa) \exp \Bigl[ v \frac{\kappa}{1 - v} \Bigr] \,d\kappa,
\end{equation}
with $M(v)$ serving as a normalization function. 
\end{defn}
The normalization function $M(v)$ in \eqref{def22} ensures the consistency of the fractional operator by recovering the original function as $v\to0^+$ and the classical first derivative as $v\to1^-$, with $M(0)=M(1)=1$. 
It also guarantees dimensional consistency and properly normalizes the exponential kernel through the factor $\frac{M(v)}{1-v}$.
Furthermore, we have $u\in H^1(a,b)$), 
and the Definition~\ref{df2} is characterized by a non-singular kernel.

\begin{defn}[\cite{Sadek2025}]\label{df3}
Let $0<\nu<1$ and $0<\alpha\le1$. 
For a function $u\in H^1(a,b)$, the hybrid Mittag-Leffler-Caputo-Fabrizio derivative is defined by
\begin{equation}\label{eq:MLCFdef}
{}^{\rm MLCF}\!D_t^{\nu,\alpha}u(t)
= \frac{M(\nu)\,\Ealpha\Bigl(\dfrac{\nu}{1-\nu}t^{\alpha}\Bigr)^{-1}}{1-\nu}
\int_0^t \Ealpha\Bigl(\dfrac{\nu}{1-\nu}\kappa^{\alpha}\Bigr) u'(\kappa)\,d\kappa,
\end{equation}
where $\Ealpha(z)=\sum_{k=0}^{\infty}\frac{z^{k}}{\Gamma(\alpha k+1)}$ denotes the one-parameter Mittag-Leffler function and $M(\nu)$ satisfies $M(0)=M(1)=1$.
\end{defn}
\begin{rem}  
For the case $\alpha=1$, we have $E_1(z)=e^{z}$ and \eqref{eq:MLCFdef} reduces to the Caputo–Fabrizio derivative.
\end{rem}

\subsection{Hybrid Fractional Model with Nonsingular Kernels}\label{sec:hybrid_model}
Here, we present the problem to be studied. To accomplish this, we extend the integer-order, deterministic susceptible-vaccinated-infected-recovered (SVIR) epidemic model presented in \cite{Kad} to a fractional-order SVIR model using the hybrid MLCF fractional derivative
\begin{subequations}\label{eq:SVIR}
\begin{align}
{}^{\rm MLCF}\!D_t^{\nu,\alpha} S(t) &= \Delta - \beta \dfrac{S I}{N} + \mu R - (k+\delta)S,
\label{eq:S}\\
{}^{\rm MLCF}\!D_t^{\nu,\alpha} V(t) &= kS - (1-\tau)\beta \dfrac{V I}{N} - \delta V,\label{eq:V}\\
{}^{\rm MLCF}\!D_t^{\nu,\alpha} I(t) &= \beta \dfrac{S I}{N} + (1-\tau)\beta \dfrac{V I}{N} - (\alpha_r+\delta+\delta_0)I,\label{eq:I}\\
{}^{\rm MLCF}\!D_t^{\nu,\alpha} R(t) &= \alpha_r I - (\delta+\mu)R, \label{eq:R}
\end{align}
\end{subequations}
where the total population $N=S+V+I+R$ and all parameters are positive constants. 
The fractional orders satisfy $0<\nu<1$ and $0<\alpha\le 1$;
the normalization factor $M(\nu)$ is usually taken as 1.
For more information on the integer-order model, including the basic reproduction number $\mathcal{R}_0$, disease-free equilibrium, and endemic equilibrium, 
we refer the interested reader to \cite{Kad}.

The introduction of the hybrid MLCF derivative allows the model to incorporate both short- and long-term memory effects within a unified non-singular framework. 
This is particularly relevant in the context of infectious diseases, where the current dynamics are often influenced by past states due to latency periods, temporary immunity, behavioral changes, and intervention delays. 
Compared with the classical integer-order SVIR model, the fractional formulation provides a more accurate description of the temporal evolution of the epidemic and can yield different transient dynamics while preserving the same equilibrium structure. 
This motivates the present study, where we develop and analyze the proposed model and construct structure-preserving numerical schemes for its efficient simulation.

\begin{table}[H]
\begin{center}
	\begin{tabular}{|c|c|}
		\hline
		parameters & Description \\
		\hline
		$\Delta$ &   birth rate \\
		$k$ & vaccination rate  \\
		$\beta$ & contact rate between infected and susceptible persons\\
		$\tau$ & vaccine effectiveness rate \\
		$\mu$  & loss of natural immunity \\ 
		$\delta_{0}$ & death rate due to Covid-19 infection \\
		$\delta$ & natural death rate\\
		$\alpha$ & rate of recovery from infection\\
		\hline    
	\end{tabular} 
     \end{center}
     \caption{Description of model parameters used in the fractional SVIR epidemic model \cite{Kad}.}
    \end{table}

\section{Dynamical Analysis of the Model}
\label{S3}
In this section we provide a concise analysis of the proposed SVIR model~\eqref{eq:SVIR}.
First, the MLCF derivative \cite{Sadek2025} of a function $u\in H^{1}(0,T)$ is defined in \eqref{eq:MLCFdef}, 
where $\Ealpha(z)=\sum_{k=0}^{\infty}\tfrac{z^{k}}{\Gamma(\alpha k+1)}$
is the one-parameter Mittag-Leffler function and $M(\nu)$ is a
normalization function satisfying $M(0)=M(1)=1$. 
We introduce the convenient notation
\begin{equation*}
   A(t) = \Ealpha\Bigl(-\frac{\nu}{1-\nu}t^{\alpha}\Bigr),\qquad
   B(t) = \Ealpha\Bigl(\frac{\nu}{1-\nu}t^{\alpha}\Bigr).
\end{equation*}
\begin{lem}[Integrated Form of the MLCF Derivative]\label{lem:integrated}
  Let $0<\nu<1$, $0<\alpha\le1$, and $u\in C([0,T])\cap H^{1}(0,T)$.
  For any $t\in[0,T]$,
  \begin{equation}\label{eq:integrated}
    u(t) = \frac{1-\nu}{M(\nu)}\,\frac{{}^{\rm MLCF}\!D^{\nu,\alpha}_tu(t)}{A(t)B(t)}
          +\frac{u(0)}{B(t)}
          +\frac{1}{B(t)}\int_0^t B'(\kappa)u(\kappa)\,d\kappa.
  \end{equation}
  Moreover, for $0<\alpha\le1$ we have
  \begin{equation}\label{eq:positivity}
    A(t)>0,\quad B(t)>0,\quad B'(\kappa)\ge0\quad\text{for all}\;\kappa \ge0.
  \end{equation}
\end{lem}
\begin{proof}
  From the definition \eqref{eq:MLCFdef} we have
  \begin{equation*}
    \int_0^t B(\kappa)u'(\kappa)\,d\tau
    =\frac{1-\nu}{M(\nu)}\frac{{}^{\rm MLCF}\!D^{\nu,\alpha}_tu(t)}{A(t)}.
  \end{equation*}
  Integration by parts yields
  \begin{equation*}
    \int_0^t B(\kappa)u'(\kappa)\,d\kappa
    = B(t)u(t) -B(0)u(0)
       -\int_0^t B'(\kappa)u(\kappa)\,d\kappa.
  \end{equation*}
  Since $B(0)=1$, solving for $u(t)$ yields \eqref{eq:integrated}.
  The positivity of $A(t)$ follows from the complete monotonicity of
  $E_{\alpha}(-x)$ on $[0,\infty)$ for $0<\alpha\le1$, cf.\  \cite{Gorenflo2014},
  while $B(t)>0$ because all terms in its series are non‑negative.
  Differentiating term-by-term gives the following for $B'(\kappa)$
  \begin{equation*}
    B'(\kappa)=\sum_{k=1}^{\infty}\frac{\bigl(\frac{\nu}{1-\nu}\bigr)^{k}
                \alpha k\,\kappa^{\alpha k-1}}{\Gamma(\alpha k+1)}\ge0,
  \end{equation*}
  as each term is non‑negative. This completes the proof.
\end{proof}

Next, we recall the following \textit{standard comparison principle} for
integro-differential equations with positive kernels.
\begin{lem}[A generalized Gronwall inequality \cite{Zhang2016}]\label{lem:gronwall}
  Let $\phi\in C([0,T])$, $\phi_{0}\ge0$, $c\ge0$, and suppose
  $\psi(t)\ge0$ is a non‑negative function. If
  \begin{equation*}
    \phi(t)\ge\phi_{0}
      - c\int_{0}^{t}\psi(\kappa)\,\phi(\kappa)\,d\kappa,\qquad t\in[0,T],
  \end{equation*}
  then $\phi(t)\ge0$ for all $t\in[0,T]$.
\end{lem}

\subsection{Positivity of Solutions}
\begin{thm}[Positivity of Solutions]\label{thm1}
  Let $(S,V,I,R)$ be the solution of system~\eqref{eq:SVIR}
  with non‑negative initial data
  \begin{equation*}
    S(0)=S_0\ge0,\quad
    V(0)=V_0\ge0,\quad
    I(0)=I_0\ge0,\quad
    R(0)=R_0\ge0.
  \end{equation*}
  Then
  \begin{equation*}
    S(t)\ge0,\quad V(t)\ge0,\quad I(t)\ge0,\quad R(t)\ge0
    \qquad\text{for all}\;t\ge0.
  \end{equation*}
\end{thm}
The proof is given in the Appendix.





\subsection{Boundedness of solutions}
\begin{thm}[Boundedness]\label{thm2}
For any $t\ge0$, the total population $N(t)=S(t)+V(t)+I(t)+R(t)$ satisfies
\begin{equation*}
0 \le N(t) \le \max\Bigl\{N(0),\,\frac{\Delta}{\delta}\Bigr\}.
\end{equation*}
Consequently, each compartment $S(t),V(t),I(t),R(t)$ is uniformly bounded.
\end{thm}
The proof is given in the Appendix.

\subsection{Local Stability Analysis}
We study the local stability of the equilibrium points of the fractional SVIR system~\eqref{eq:SVIR}
with the total population $N=S+V+I+R$. 
All parameters are positive constants.  
The MLCF derivative is a non‑singular fractional operator with a completely positive kernel;
for such operators the stability of an equilibrium is determined by the eigenvalues of the Jacobian matrix of the right‑hand side $J$ evaluated at the equilibrium \cite{StabilityFractional}. 
Specifically, an equilibrium $\mathbf{x}^*$ is \textit{locally asymptotically stable}, if every eigenvalue $\lambda$ of $J(\mathbf{x}^*)$ satisfies \cite{Matignon1996}
 $|\arg(\lambda)| > \frac{\nu\pi}{2},$ 
and unstable if there exists an eigenvalue with $|\arg(\lambda)| < \frac{\nu\pi}{2}$.

At the equilibrium, all derivatives vanish. Adding the four equations \eqref{eq:S}--\eqref{eq:R} gives:
\begin{equation*}
    \Delta = \delta N + \delta_0 I.
\end{equation*}
First, we focus on the \textit{disease‑free equilibrium} (DFE).
Setting $I=0$ and $R=0$ yields
\begin{equation*}
   S_0 = \frac{\Delta}{k+\delta},\qquad 
   V_0 = \frac{k S_0}{\delta}= \frac{k\Delta}{\delta(k+\delta)},\qquad 
   N_0 = S_0+V_0 = \frac{\Delta}{\delta}.
\end{equation*}
Thus the DFE is $E_0=(S_0,V_0,0,0)$.  

The \textit{basic reproduction number} $\mathcal{R}_0$ is obtained via the next‑generation method \cite{Van}.
The infection rate from $S$ and $V$ is $\beta\frac{SI}{N}+(1-\tau)\beta\frac{VI}{N}$, and the removal rate from $I$ is $\alpha_r+\delta+\delta_0$.
At the DFE, we have
\begin{equation} \label{bas}
\mathcal{R}_0 = \frac{\beta}{\alpha_r+\delta+\delta_0}\,\frac{S_0+(1-\tau)V_0}{N_0}
= \frac{\beta}{\alpha_r+\delta+\delta_0}\Bigl( \frac{1}{k+\delta} + \frac{(1-\tau)k}{\delta(k+\delta)} \Bigr).
\end{equation}

Next, we turn to consider the \textit{endemic equilibrium} (EE).
When $\mathcal{R}_0>1$, a unique positive equilibrium $E^*=(S^*,V^*,I^*,R^*)$ exists. 
From the equilibrium conditions we obtain
\begin{align*}
   \Delta &= \beta\frac{S^*I^*}{N^*} - \mu R^* + (k+\delta)S^*, \\[2pt]
   k S^* &= (1-\tau)\beta\frac{V^*I^*}{N^*} + \delta V^*, \\[2pt]
    \beta\frac{S^*+(1-\tau)V^*}{N^*} &= \alpha_r+\delta+\delta_0, \\[2pt]
    R^* &= \frac{\alpha_r}{\delta+\mu} I^*.
\end{align*}
The total population satisfies $\Delta = \delta N^* + \delta_0 I^*$. Explicit formulas are not needed for the stability analysis.

\begin{thm}[Stability of the DFE]\label{thm3}
The disease‑free equilibrium $E_0$ is locally asymptotically stable if $\mathcal{R}_0<1$ and unstable if $\mathcal{R}_0>1$.
\end{thm}
The proof is given in the Appendix.

Assume $\mathcal{R}_0>1$ so that a unique positive endemic equilibrium $E^*$ exists. 
To analyze its stability we compute the Jacobian at $E^*$ and examine its eigenvalues.
For convenience, we define
\begin{equation*}
  u = \frac{S^*}{N^*},\quad v = \frac{V^*}{N^*},
  \quad i = \frac{I^*}{N^*},\quad r = \frac{R^*}{N^*}.
\end{equation*}
Then $u+v+i+r=1$. From the equilibrium equations we obtain
\begin{equation*}
   p := \beta\frac{I^*}{N^*} = \beta i,\qquad q := \beta\frac{S^*}{N^*} = \beta u,\qquad r := \beta\frac{V^*}{N^*} = \beta v,
\end{equation*}
and the relation $q + (1-\tau)r = \alpha_r+\delta+\delta_0$.

\begin{thm}[Stability of the EE]\label{thm4}
If $\mathcal{R}_0>1$, the endemic equilibrium $E^*$ exists and is locally asymptotically stable.
\end{thm}
The proof is given in the Appendix.

We have shown that the DFE 
is locally asymptotically stable when $\mathcal{R}_0<1$ and becomes unstable when $\mathcal{R}_0>1$, at which point a stable endemic equilibrium emerges. 
This threshold behaviour is typical of epidemic models and demonstrates that fractional‑order memory does not change the stability conditions; it only affects the rate of convergence through the fractional exponents $\nu$ and $\alpha$.

\subsection{Global Stability Analysis of the Fractional SVIR Model }
The MLCF derivative can be written as an integral with a positive kernel. Consequently, it satisfies a comparison principle: 
if ${}^{\rm MLCF}\!D_t^{\nu,\alpha} x(t) \le{}^{\rm MLCF}\!D_t^{\nu,\alpha} y(t)$ and $x(0)\le y(0)$, then $x(t)\le y(t)$ for all $t\ge0$. Moreover, the following inequality holds.

\begin{lem}[Chain‑rule Inequality]\label{lem:chain}
Let $0<\nu<1$, $0<\alpha\le1$ and let $x\colon[0,T]\to(0,\infty)$ be continuously differentiable. 
Define $\varphi(u)=u-1-\ln u\ge0$ for $u>0$. 
Then for any constant $c>0$,
\begin{equation*}
{}^{\rm MLCF}\!D_t^{\nu,\alpha} \Bigl[c\,\varphi\Bigl(\frac{x}{c}\Bigr)\Bigr] 
\le \Bigl(1-\frac{c}{x}\Bigr){}^{\rm MLCF}\!D_t^{\nu,\alpha}  x.
\end{equation*}
\end{lem}
\begin{proof}
The proof uses the convexity of $\varphi$ and the positivity of the kernel; it is given in Lemma~\ref{lem:gronwall}.
\end{proof}

\begin{lem}[Fractional LaSalle Principle \cite{rathee2025sensitivity}]\label{lem:lasalle}
Let $\Omega\subset\mathbb{R}^n$ be a compact positively invariant set for the system ${}^{\rm MLCF}\!D_t^{\nu,\alpha} x = f(x)$, where $f$ is continuous. 
Let $V\colon\Omega\to\mathbb{R}$ be continuously differentiable and satisfy ${}^{\rm MLCF}\!D_t^{\nu,\alpha} V(x) \le 0$ for all $x\in\Omega$. 
Then every solution starting in $\Omega$ converges to the largest invariant subset of $\{x\in\Omega\colon{}^{\rm MLCF}\!D_t^{\nu,\alpha} V(x)=0\}$.
\end{lem}

We define
\begin{equation*}
    \Omega = \{(S,V,I,R)\in\mathbb{R}_{\ge0}^4\colon N\le M\},\qquad M = \max\{N(0),\Delta/\delta\}.
\end{equation*}
Using the total population equation
\begin{equation*}
    {}^{\rm MLCF}\!D_t^{\nu,\alpha} N = \Delta - \delta N - \delta_0 I \le \Delta - \delta N,
\end{equation*}
the comparison principle gives $N(t)\le M$ for all $t$. 
Non‑negativity of the variables follows from standard arguments for quasi‑positive systems. 
Hence $\Omega$ is compact and positively invariant.

Setting the right‑hand sides of the system~\eqref{eq:SVIR}
to zero yields the 
DFE 
\begin{equation*}
E_0 = (S_0,V_0,0,0),\quad S_0 = \frac{\Delta}{k+\delta},\quad V_0 = \frac{k\Delta}{\delta(k+\delta)},\quad N_0 = \frac{\Delta}{\delta}.
\end{equation*}
For the basic reproduction number \eqref{bas},
if $\mathcal{R}_0>1$, there exists a unique endemic equilibrium $E^*=(S^*,V^*,I^*,R^*)$ with all components positive, satisfying
\begin{subequations}\label{Equilibrium}
\begin{align}
\Delta &= \beta\frac{S^*I^*}{N^*} - \mu R^* + (k+\delta)S^*, \label{E1}\\
kS^* &= (1-\tau)\beta\frac{V^*I^*}{N^*} + \delta V^*, \label{E2}\\
\beta\frac{S^*I^*}{N^*} + (1-\tau)\beta\frac{V^*I^*}{N^*} &= (\alpha_r+\delta+\delta_0)I^*, \label{E3}\\
\alpha_r I^* &= (\delta+\mu)R^*. \label{E4}
\end{align}
\end{subequations}

\begin{thm}[Global stability of the DFE]\label{thm:DFE}
If $\mathcal{R}_0\le 1$, then $E_0$ is globally asymptotically stable in $\Omega$.
\end{thm}
The proof is given in the Appendix.

\begin{thm}[Global stability of the EE]\label{thm:EE}
If $\mathcal{R}_0>1$, then the endemic equilibrium $E^*$ is globally asymptotically stable in $\interior(\Omega)$.
\end{thm}
The proof is given in the Appendix.


\section{Numerical Approximation of the MLCF Derivative}\label{S4}
The numerical approximation of the MLCF derivative at $t = t_n$ using the two point forward difference and trapezoidal rule is given by \cite{Sadek2025}:
\begin{equation*}
^{\rm MLCF}D_t^{\nu,\alpha}[u(t)] \big|_{t=t_n} 
\approx \frac{M(\nu)}{1-\nu} \, \Ealpha\Bigl( -\frac{\nu}{1-\nu}\, t_n^{\alpha} \Bigr) 
\sum_{k=1}^{n} \frac{u_k - u_{k-1}}{h} 
\int_{t_{k-1}}^{t_k} \Ealpha\Bigl( \frac{\nu}{1-\nu} \tau^{\alpha} \Bigr)\, d\tau.
\end{equation*}
Approximating the integral via the trapezoidal rule yields
\begin{equation*}
   \int_{t_{k-1}}^{t_k} \Ealpha\Bigl( \frac{\nu}{1-\nu} \tau^{\alpha} \Bigr) \,d\tau 
\approx \frac{h}{2} \biggl[ \Ealpha\Bigl( \frac{\nu}{1-\nu} t_k^{\alpha} \Bigr) 
+ \Ealpha\Bigl( \frac{\nu}{1-\nu} t_{k-1}^{\alpha} \Bigr) \biggr].
\end{equation*}
Substituting this into the derivative approximation gives the final discrete scheme:
\begin{align}
^{\rm MLCF}D_t^{\nu,\alpha}[u(t)] \big|_{t=t_n} &\approx \sum_{k=1}^{n} (u_k - u_{k-1}) W_{k,n},
\end{align}
with the weights
\begin{equation*}
W_{k,n} = \frac{M(\nu) \, h}{2(1-\nu)} \, \Ealpha\Bigl( -\frac{\nu}{1-\nu} \,t_n^{\alpha} \Bigr) 
\biggl[ \Ealpha\Bigl( \frac{\nu}{1-\nu} t_k^{\alpha} \Bigr) 
+ \Ealpha\Bigl( \frac{\nu}{1-\nu} t_{k-1}^{\alpha} \Bigr) \biggr].
\end{equation*}
In practice, we evaluate the Mittag‑Leffler function with a positive argument, defining
\begin{equation*}
    E_n = E_\alpha\Bigl(\frac{\nu}{1-\nu}\,t_n^\alpha\Bigr).
\end{equation*}
Then, the approximation can be rewritten as
\begin{equation}
   ^{\rm MLCF}D_t^{\nu,\alpha} u(t_n) \approx 
   \frac{1}{P_n}\sum_{k=1}^{n} (u_k-u_{k-1}) \,w_k,
   \label{pw1}
\end{equation}
where we have introduced
\begin{equation}
w_n = E_n + E_{n-1}, \qquad 
P_n = \frac{M(\nu)}{2(1-\nu)E_n}.
\end{equation}
Separating the term $k=n$ and defining the \textit{memory variable}
\begin{equation*}
     \mem_{n-1} = \sum_{k=1}^{n-1} (u_k-u_{k-1}) \,w_k,
\end{equation*}
we obtain the compact form
\begin{equation}\label{eq:mlcf_compact}
    ^{\rm MLCF}D_t^{\nu,\alpha} u(t_n) \approx \frac{1}{P_n}\bigl( u_n w_n - (u_{n-1}w_{n-1} + \mem_{n-1})\bigr).
\end{equation}
\begin{rem}
In the special case $\alpha=1$, the Mittag-Leffler function reduces to the exponential function, i.e., 
$E_1(z) = e^z $, and the scheme simplifies to the discretization of the Caputo-Fabrizio derivative:
\begin{equation*}
   ^{\rm CF}D_t^{\nu}[u(t)] \big|_{t=t_n} 
   \approx \frac{M(\nu) \, h}{2(1-\nu)} \, \mathrm{e}^{ -\frac{\nu}{1-\nu} t_n } 
    \sum_{k=1}^{n} (u_k - u_{k-1}) 
    \Bigl[ \mathrm{e}^{ \frac{\nu}{1-\nu} t_k } + \mathrm{e}^{ \frac{\nu}{1-\nu} t_{k-1} } \Bigr].
\end{equation*}
\end{rem}

\subsection{Nonstandard Finite Difference Discretisation}
The NSFD method replaces the standard denominator $h$ by a function $\phi(h)$ that satisfies $\phi(h)=h+O(h^2)$ and enhances the qualitative behaviour of the numerical solution. 
Following \cite{mickens}, we choose
\begin{equation*}
    \phi(h) = 1 - \mathrm{e}^{-h}.
\end{equation*}
For a fractional differential equation $^{\rm MLCF}D_t^{\nu,\alpha} u = f(u)$, we impose the discrete analogue
\begin{equation*}
    \frac{u_n - u_{n-1}}{\phi(h)} = \frac{1}{P_n w_n}\bigl( f(u_n) - P_n \mem_{n-1} \bigr),
\end{equation*}
which is obtained by substituting \eqref{eq:mlcf_compact} and solving for the derivative. Solving for $u_n$ yields an implicit update formula.

To allow flexibility in the time discretization, we approximate the right‑hand side $f(u)$ by a \textit{weighted average} of its values at the current and previous time steps 
\begin{equation*}
   f\bigl(u(t)\bigr) \approx \theta f(u_n) + (1-\theta) f(u_{n-1}), \quad \theta\in[0,1].
\end{equation*}
The case $\theta=1$ gives the fully implicit scheme used in earlier work; $\theta=0$ corresponds to an explicit scheme; 
$\theta=1/2$ is the Crank–Nicolson type, which is second‑order accurate for ordinary differential equations.

Incorporating this weighted average into the NSFD discretization leads to the following implicit relation:
\begin{equation}\label{disc1}
\frac{u_n - u_{n-1}}{\phi(h)} = \frac{1}{P_n w_n}\Bigl( \theta f(u_n) + (1-\theta) f(u_{n-1}) - P_n \mem_{n-1} \Bigr).
\end{equation}
Rearranging, we obtain the nonlinear system for $u_n$:
\begin{equation}\label{eq:theta_implicit}
u_n - u_{n-1} - \frac{\phi}{P_n w_n}\Bigl( \theta f(u_n) + (1-\theta) f(u_{n-1}) - P_n \mem_{n-1} \Bigr) = 0.
\end{equation}

We define the residual vector $\mathbf{r}(\mathbf{x})$ for the four state variables $\mathbf{x}=[S,V,I,R]^\top$ at time step $m$:
\begin{equation*}
   \mathbf{r}(\mathbf{x}) = \mathbf{x} - \mathbf{x}_{m-1} 
   - \frac{\phi}{P_m w_m}\Bigl( \theta \mathbf{f}(\mathbf{x}) + (1-\theta)\mathbf{f}(\mathbf{x}_{m-1}) - P_m \mathbf{mem}_{m-1} \Bigr),
\end{equation*}
where $\mathbf{f}=[f_1,f_2,f_3,f_4]^\top$ with
\begin{align*}
f_1(S,V,I,R) &= \Delta - \beta\frac{SI}{N} + \mu R - (k+\delta)S,\\
f_2(S,V,I,R) &= k S - (1-\tau)\beta\frac{VI}{N} - \delta V,\\
f_3(S,V,I,R) &= \beta\frac{SI}{N} + (1-\tau)\beta\frac{VI}{N} - (\alpha_r+\delta+\delta_0)I,\\
f_4(S,V,I,R) &= \alpha_r I - (\delta+\mu)R,
\end{align*}
and $N=S+V+I+R$ (with a small safeguard if $N$ becomes zero).
The memory term $\mathbf{mem}_{m-1}$ is a vector containing the four accumulated memory values from previous steps, defined for each variable as in \eqref{eq:mlcf_compact}.

To apply Newton’s method we need the Jacobian $\mathbf{J}_r(\mathbf{x}) = \partial \mathbf{r}/\partial \mathbf{x}$, from the definition,
\begin{equation*}
    \mathbf{J}_r(\mathbf{x}) = \mathbf{I} - \frac{\phi}{P_m w_m}\,\theta\,\mathbf{J}_f(\mathbf{x}),
\end{equation*}
where $\mathbf{J}_f(\mathbf{x}) = \partial \mathbf{f}/\partial \mathbf{x}$ is the Jacobian of $\mathbf{f}$. 
A maximum of 30 iterations is allowed; if convergence is not reached, the last iterate is accepted. In practice, for $\theta>0$ the Newton method usually converges in 3--5 iterations.
%
After the converged solution $\mathbf{x}_m = [S_m,V_m,I_m,R_m]^\top$ is obtained, the memory terms are updated for use in the next time step:
\begin{align*}
\mem_{S,m} &= \mem_{S,m-1} + (S_m - S_{m-1}) w_m, \\
\mem_{V,m} &= \mem_{V,m-1} + (V_m - V_{m-1}) w_m, \\
\mem_{I,m} &= \mem_{I,m-1} + (I_m - I_{m-1}) w_m, \\
\mem_{R,m} &= \mem_{R,m-1} + (R_m - R_{m-1}) w_m.
\end{align*}

\begin{figure}[htbp]
\includegraphics[width=1.0\textwidth]{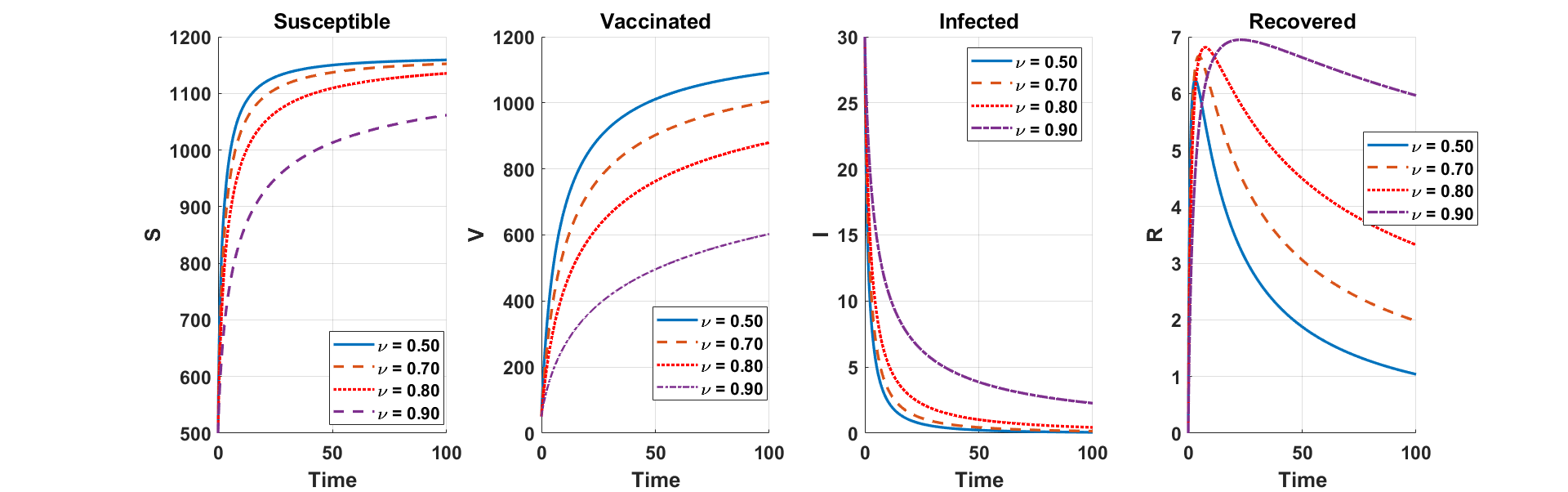}
	\caption{Numerical solution for $
    \alpha=0.8,$ $\theta=1$,  and different $\nu$.}
\label{ch1}
\end{figure}
\begin{figure}[htbp]
\centering
\includegraphics[width=1.0\textwidth]{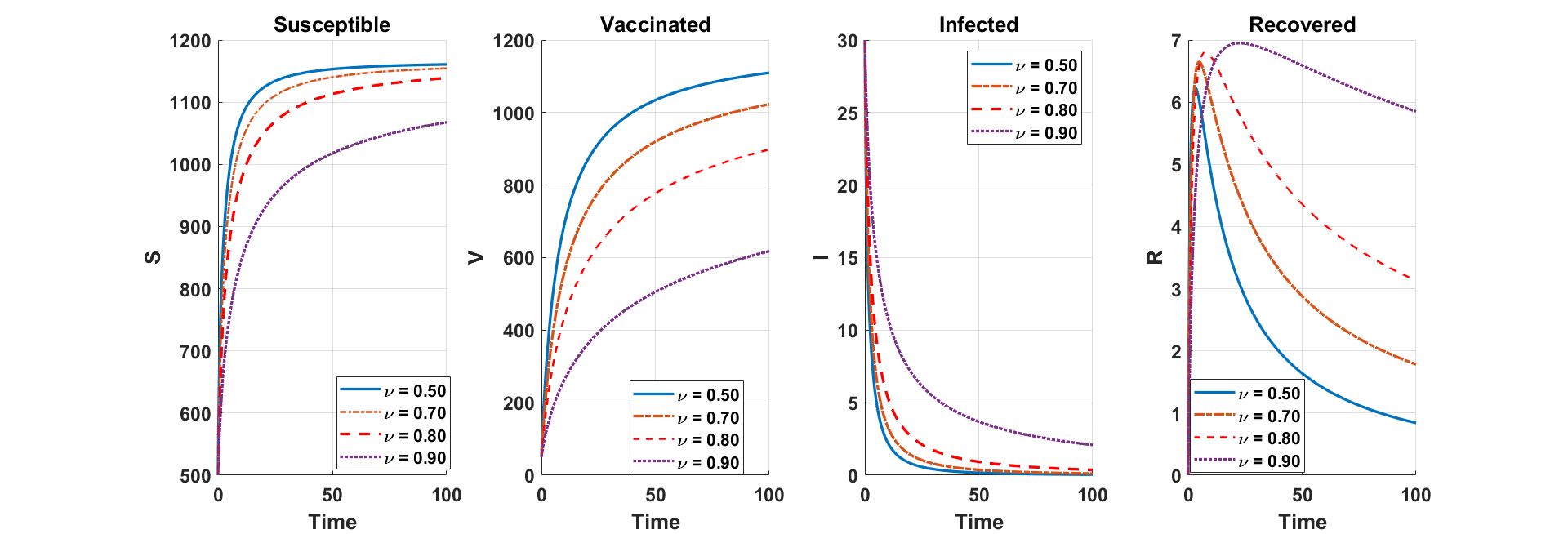}
	\caption{Numerical solution for $
    \alpha=1,$ $\theta=1$,  and different $\nu$. }
\label{ch2}
\end{figure}
\begin{figure}[htbp]
\centering
\includegraphics[width=1.0\textwidth]{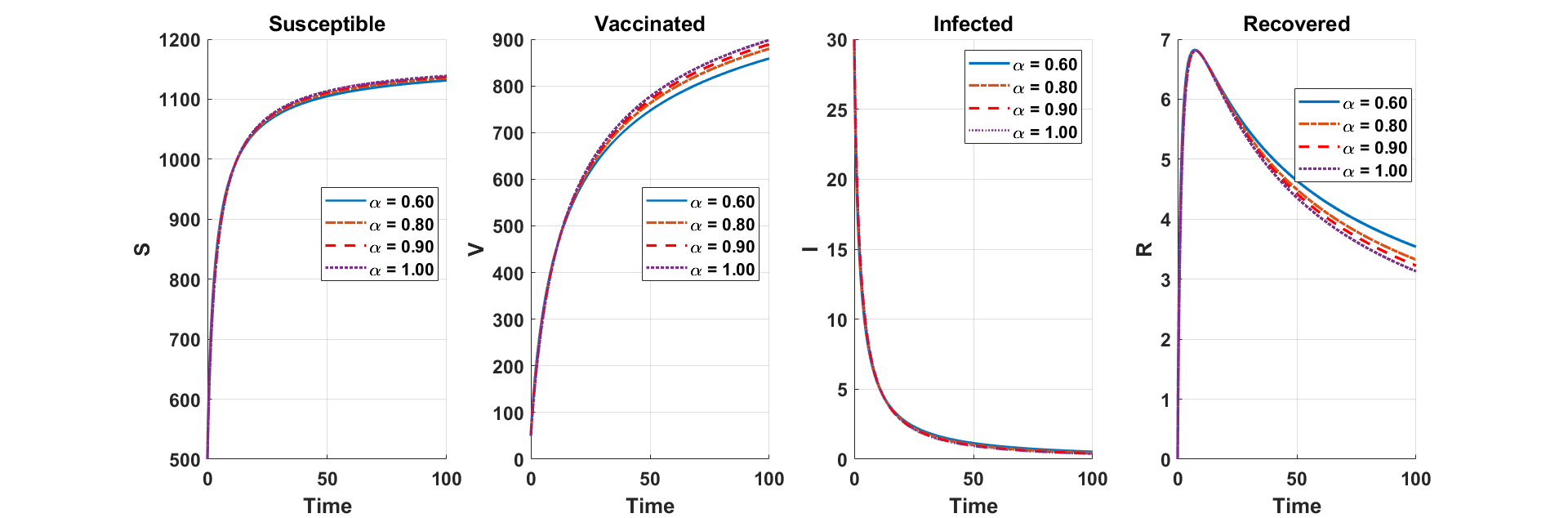}
	\caption{Numerical solution for $\nu=0.8$,  $\theta=1$,   and different $\alpha$.}
\label{ch3}
\end{figure}
\begin{figure}[htbp]
\includegraphics[width=1.0\textwidth]{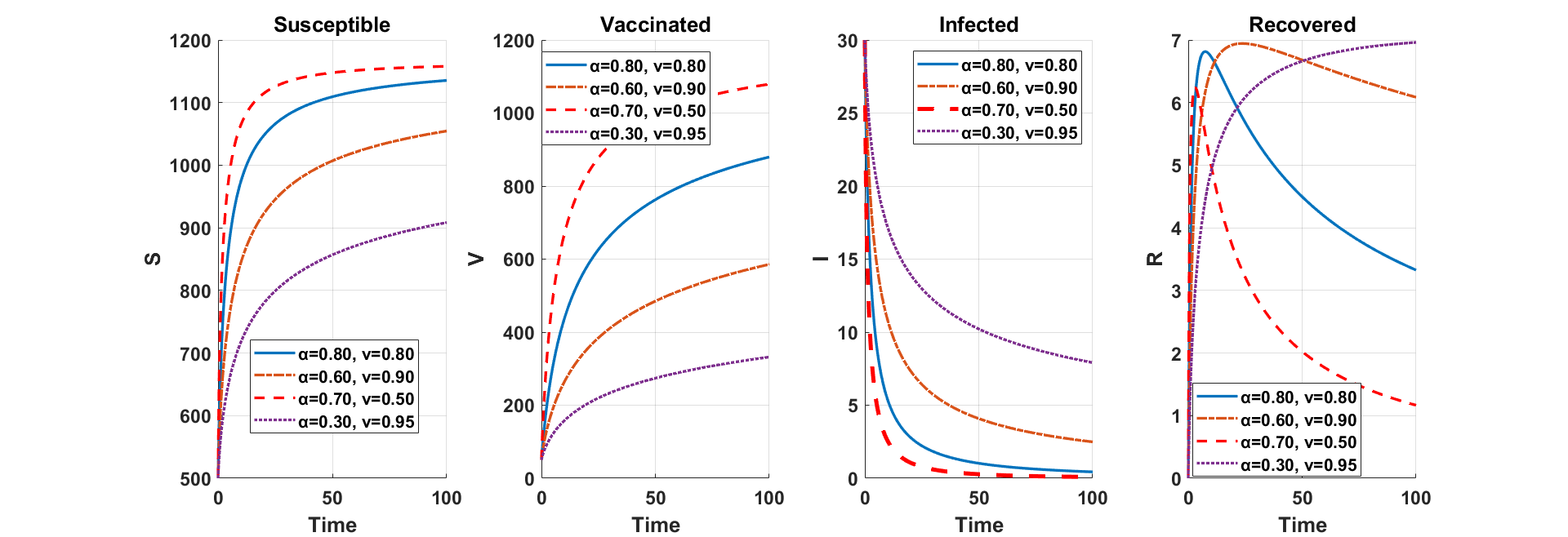}
	\caption{ Numerical solution for $\theta=1,$ and  different $\alpha, \nu$.}
\label{ch4}
\end{figure}

\begin{figure}[htbp]
\includegraphics[width=1.\textwidth]{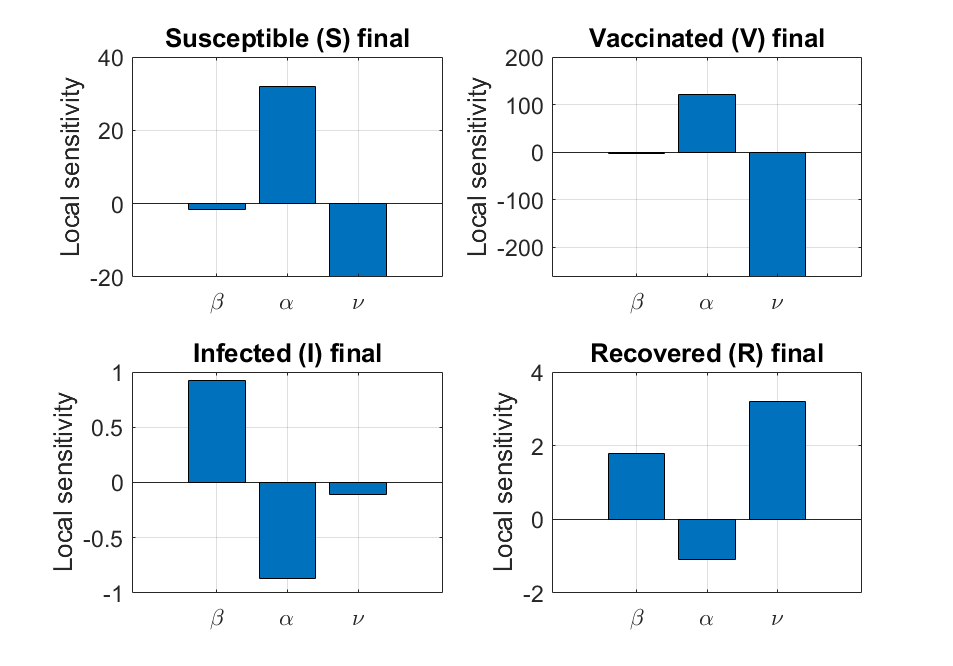}
	\caption{Local sensitivity of final compartment values ($1\% $ perturbation). }
\label{ch5}
\end{figure}
\begin{figure}[htbp]
\includegraphics[width=1.\textwidth]{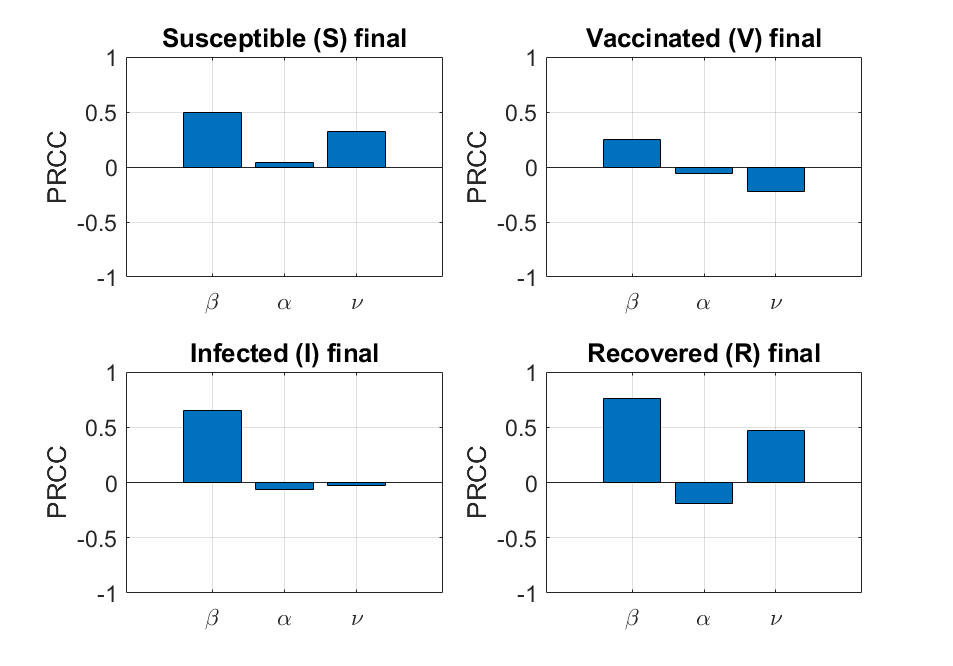}
	\caption{Global sensitivity (PRCC) of final compartment values. }
\label{ch6}
\end{figure}
\begin{figure}[htbp]
\includegraphics[width=0.5\textwidth]{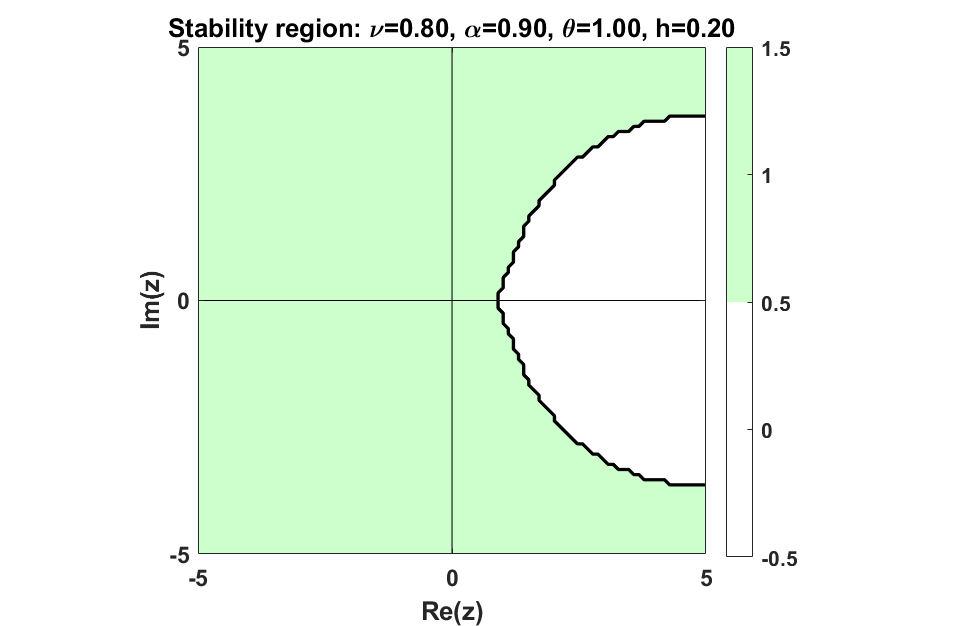}
\includegraphics[width=0.5\textwidth]{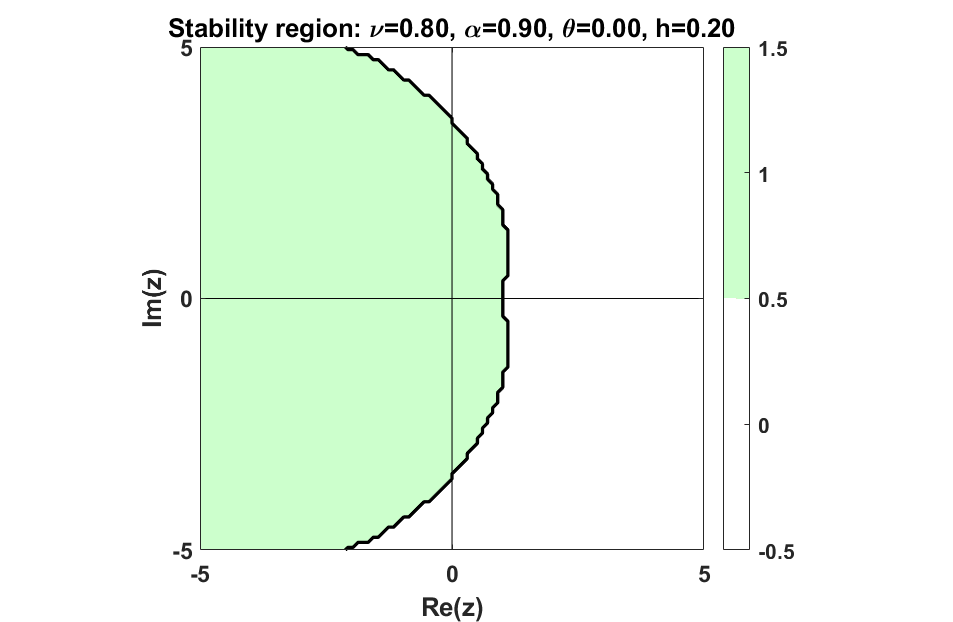}
\includegraphics[width=0.5\textwidth]{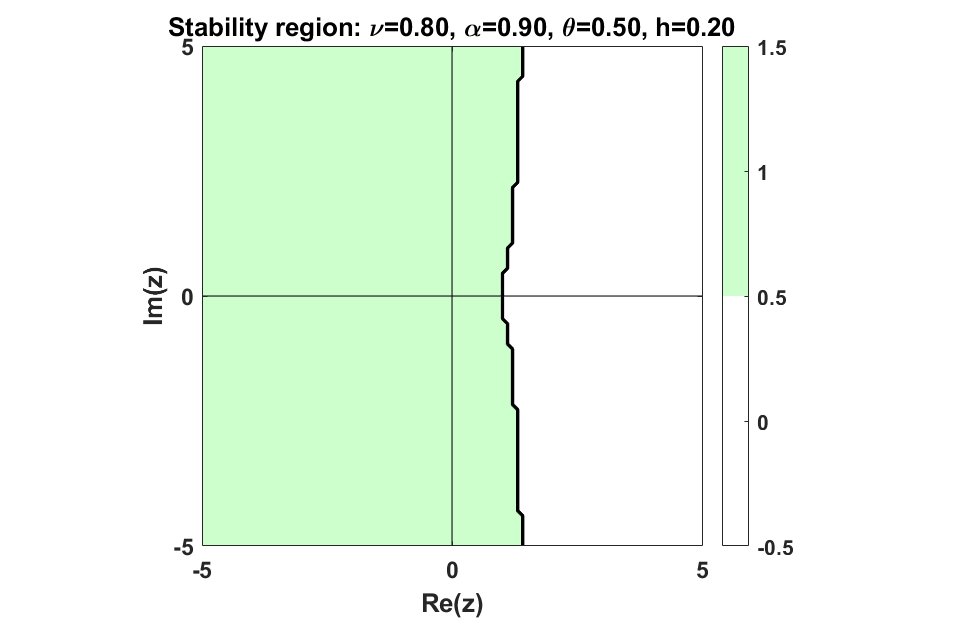}
\includegraphics[width=0.5\textwidth]{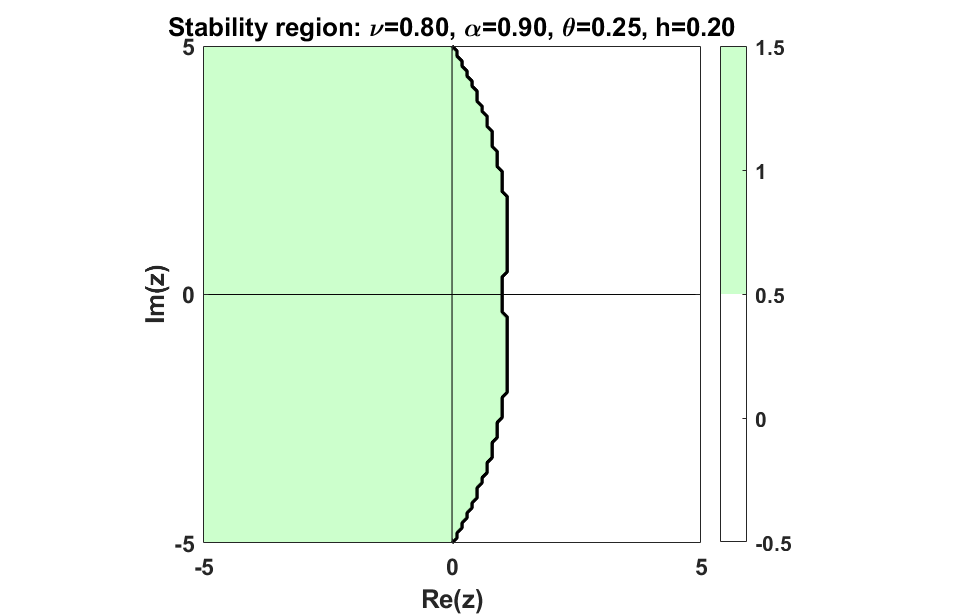}
\includegraphics[width=0.5\textwidth]{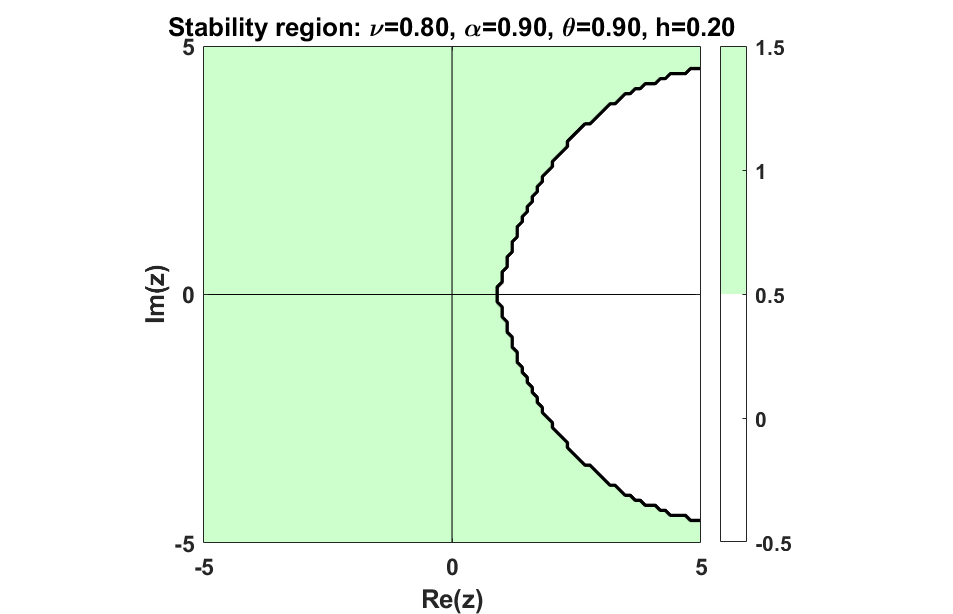}
\includegraphics[width=0.5\textwidth]{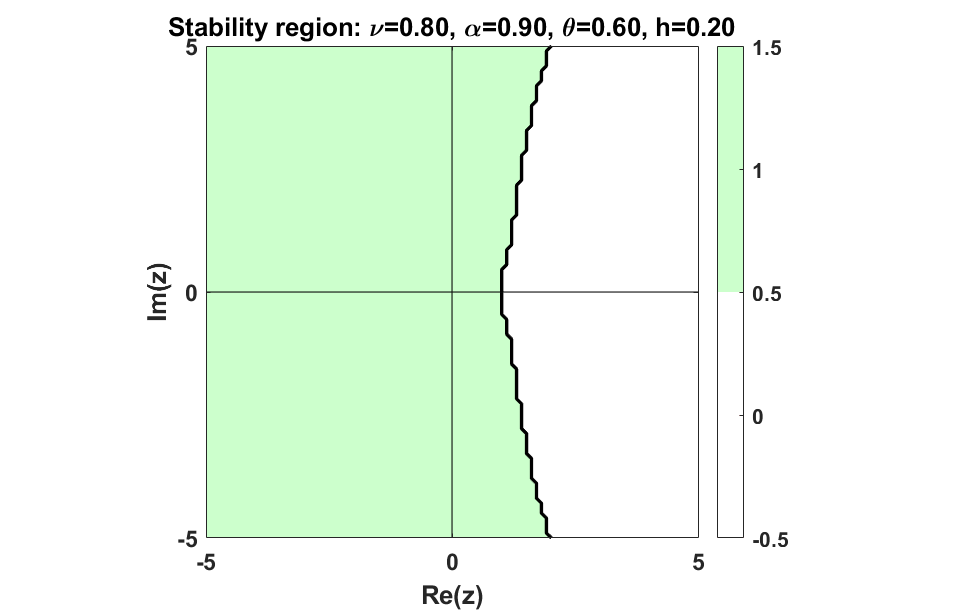}
	\caption{ Region Stability for $\nu=0.8,$ and $\alpha=0.9$ and different values of $\theta$}
\label{ch7}
\end{figure}
\begin{figure}[htbp]
\includegraphics[width=0.5\textwidth]{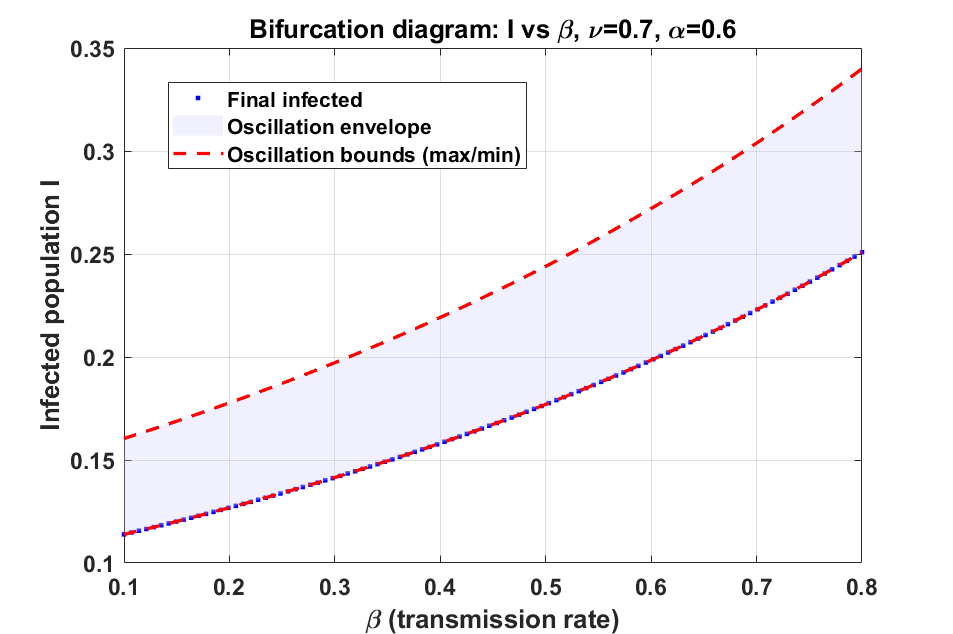}
\includegraphics[width=0.5\textwidth]{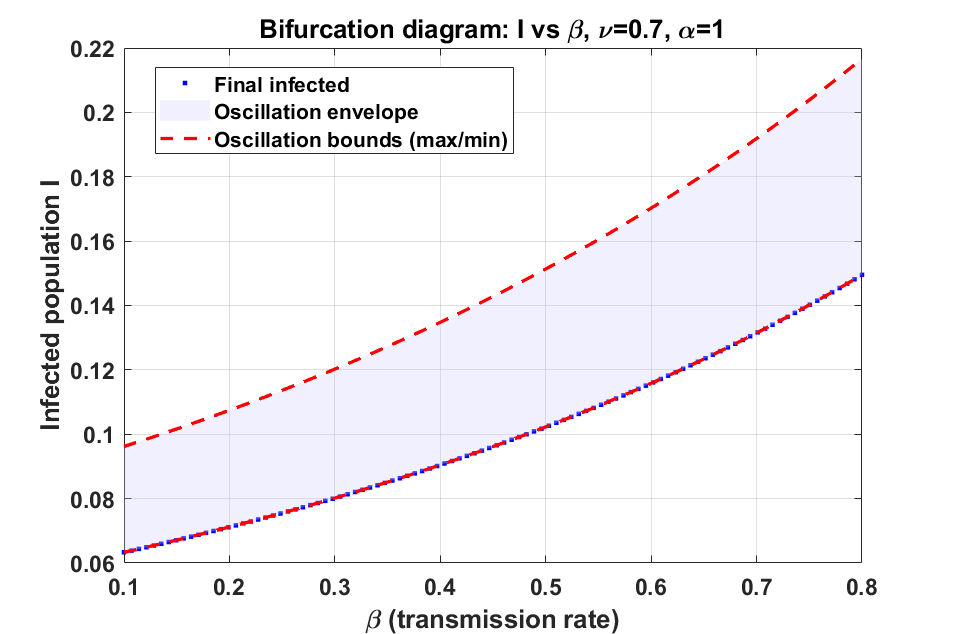}
\includegraphics[width=0.5\textwidth]{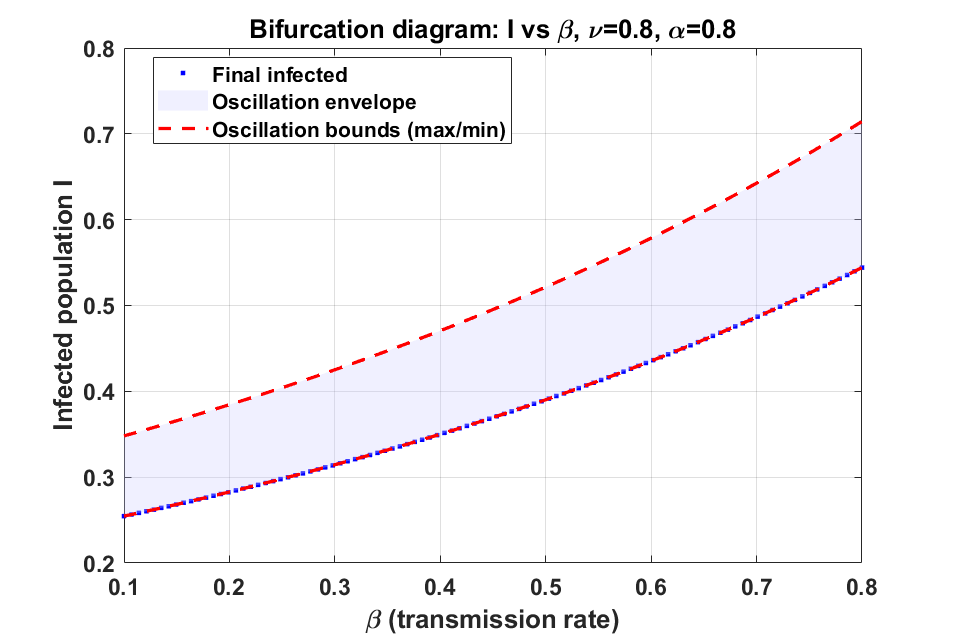}
\includegraphics[width=0.5\textwidth]{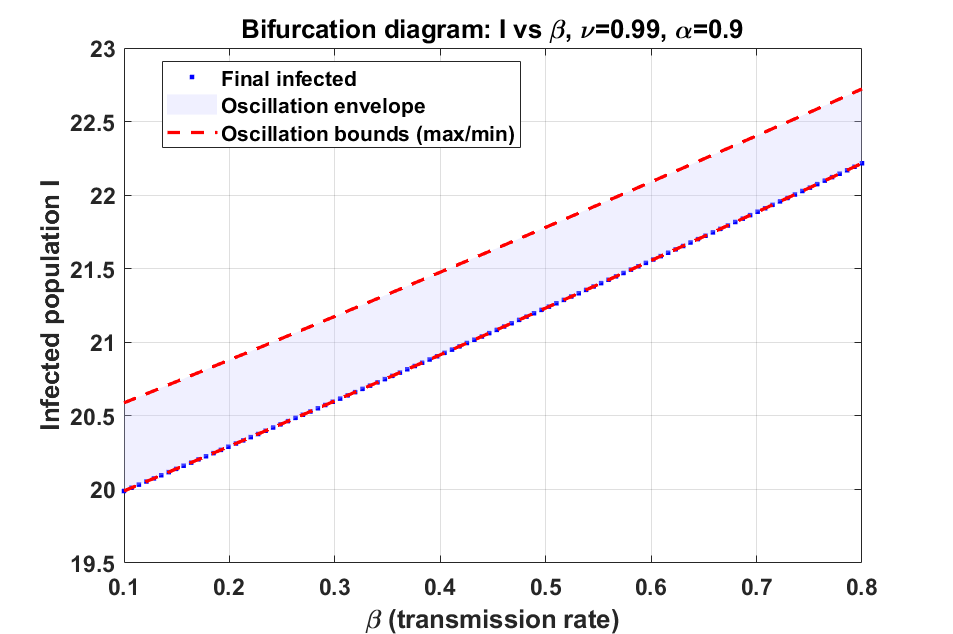}
	\caption{ The bifurcation diagram $I$ vs $\beta$ at different values of $\nu$ and $\alpha$ and $\theta=1$.}
\label{ch8}
\end{figure}
\begin{figure}[htbp]
\includegraphics[width=1.0\textwidth]{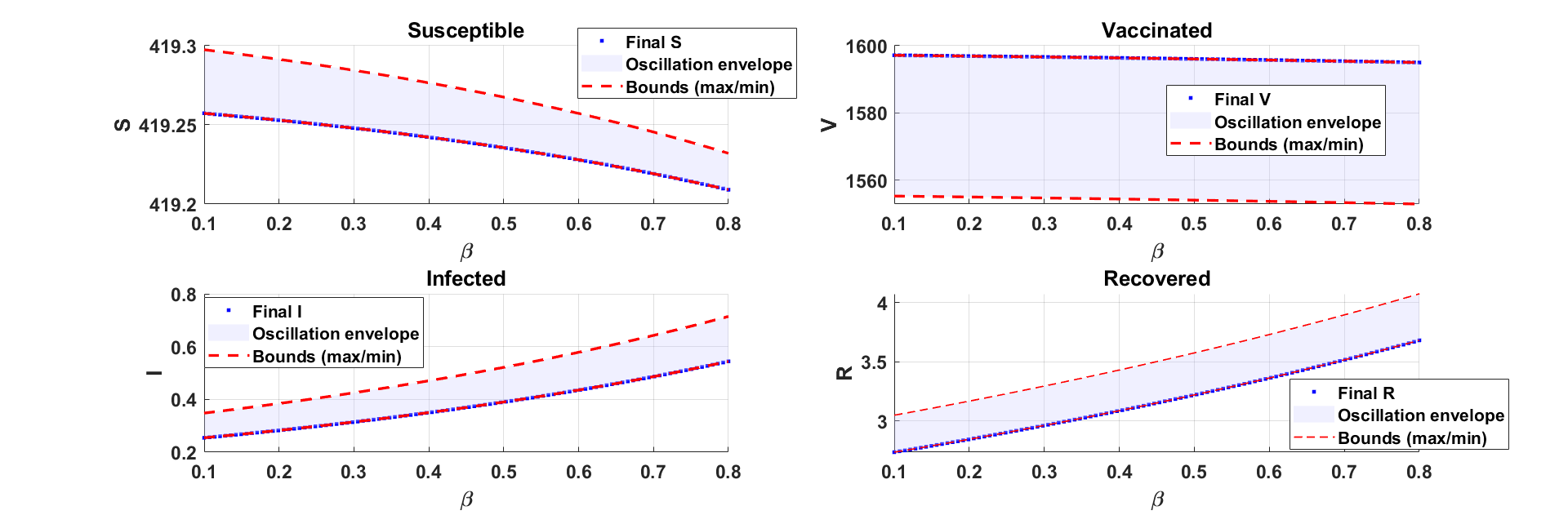}
	\caption{ The bifurcation diagram $S,I,V,R$ vs $\beta$ at $\nu=0.8$ and $\alpha=0.8$ and $\theta=1$.}
\label{ch9}
\end{figure}
\begin{figure}[htbp]
\centering
\includegraphics[width=0.8\textwidth]{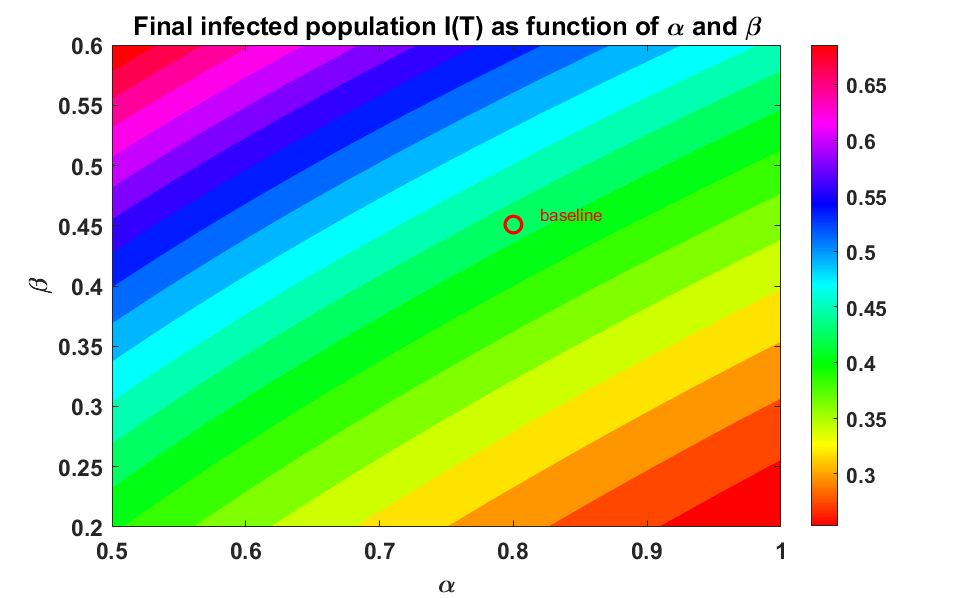}
	\caption{Final infected population $I(T)$ at $\nu=0.8$ and $\alpha=0.8$ and $\theta=1$}
\label{ch10}
\end{figure}

\section{Stability Analysis of the Proposed Method}\label{S5}
We consider a numerical scheme for the fractional SVIR model based on the hybrid MLCF derivative. 
The corresponding discretization is given by \eqref{disc1},
with 
\begin{equation*}
\phi(h)=1-\mathrm{e}^{-h},\quad
P_n = \frac{M(\nu)}{2(1-\nu)E_\alpha\bigl(\frac{\nu}{1-\nu}t_n^\alpha\bigr)},\quad w_n = E_\alpha\bigl(\frac{\nu}{1-\nu}t_n^\alpha\bigr)+E_\alpha\bigl(\frac{\nu}{1-\nu}t_{n-1}^\alpha\bigr),
\end{equation*}
and $\mem_{n-1}$ the history accumulator. 
In this note we prove the main stability properties of the scheme.

First, we consider the Dahlquist test equation 
$^{\rm MLCF}D_t^{\nu,\alpha} u(t) = \lambda u(t)$,  
 $\Real(\lambda)\le0$,
which serves as a proxy for the linearized system.
For the fully implicit case ($\theta=1$) the scheme becomes
\begin{equation*}
  \frac{u_n-u_{n-1}}{\phi}
  = \frac{1}{P_n w_n}\bigl(\lambda u_n - P_n\mem_{n-1}\bigr).
\end{equation*}
Solving for $u_n$ yields:
\begin{equation*}
   u_n - \frac{\phi}{P_n w_n}\lambda u_n 
   = u_{n-1} - \frac{\phi}{w_n}\mem_{n-1},
\end{equation*}
\begin{equation*}
u_n\Bigl(1 - \frac{\phi}{P_n w_n}\lambda\Bigr) = u_{n-1} - \frac{\phi}{w_n}\mem_{n-1}.
\end{equation*}
Hence
\begin{equation*}
u_n = \frac{1}{1 - \frac{\phi}{P_n w_n}\lambda}\,u_{n-1} - \frac{\phi}{w_n}\frac{1}{1 - \frac{\phi}{P_n w_n}\lambda}\,\mem_{n-1}.
\end{equation*}
Let us define $A_n = \frac{\phi}{P_n w_n}>0$. Then, we obtain
\begin{equation*}
    u_n = \frac{1}{1 - A_n\lambda}\,u_{n-1} - \frac{\phi}{w_n}\frac{1}{1 - A_n\lambda}\,\mem_{n-1}.
\end{equation*}
For typical epidemic models the linearization yields eigenvalues $\lambda$ that are real and non‑positive ($\lambda \le 0$). In that case $1 - A_n\lambda \ge 1$, so
\begin{equation*}
0 < \frac{1}{1 - A_n\lambda} \le 1.
\end{equation*}
Thus the coefficient of $u_{n-1}$ does not amplify the solution. 
Moreover, the memory term involves only past values and its coefficient is bounded. 
By induction, if all previous values are bounded, $u_n$ remains bounded. 
Therefore the scheme is unconditionally stable for $\lambda\le0$ when $\theta = 1$.
For $\theta < 1$ the stability region may be restricted; however, the paper employs $\theta = 1$ in all simulations, ensuring unconditional linear stability.

In the sequel we prove rigorously for the fully implicit ($\theta=1$) NSFD scheme that
  the numerical solution remains non‑negative for all time steps if the initial data are non‑negative.
  Also, we show that the total population and each compartment are uniformly bounded by $\max(N_0,\Delta/\delta)$.
These properties mirror the qualitative behaviour of the continuous fractional SVIR model.

\subsection{Positivity of the NSFD Solution}
For a fixed time step $h>0$ we set $t_n = n h$, $n=0,1,\dots$. 
The NSFD discretization of the MLCF derivative at $t_n$ is
\begin{equation*}
   ^{\rm MLCF}D_t^{\nu,\alpha}u(t_n) \approx \frac{1}{P_n}\sum_{k=1}^{n} (u_k-u_{k-1}) w_k,
\end{equation*}
with
\begin{equation}
   w_k = E_\alpha\Bigl(\frac{\nu}{1-\nu}t_k^\alpha\Bigr)
        +E_\alpha\Bigl(\frac{\nu}{1-\nu}t_{k-1}^\alpha\Bigr),\qquad
P_n = \frac{M(\nu)}{2(1-\nu)\,E_\alpha\bigl(\frac{\nu}{1-\nu}t_n^\alpha\bigr)}.
\label{WP}
\end{equation}
All $P_n$ and $w_k$ are strictly positive because the Mittag‑Leffler function is positive for positive arguments. The NSFD denominator is $\phi(h)=1-\mathrm{e}^{-h}>0$.

For the fully implicit case ($\theta=1$), the discrete SVIR system becomes
\begin{subequations}\label{discrete_SVIR}
\begin{align}
\frac{1}{P_n}\sum_{k=1}^{n}(S_k-S_{k-1})w_k &= f_S(S_n,V_n,I_n,R_n),\label{1s}\\
\frac{1}{P_n}\sum_{k=1}^{n}(V_k-V_{k-1})w_k &= f_V(S_n,V_n,I_n,R_n),\label{2s}\\
\frac{1}{P_n}\sum_{k=1}^{n}(I_k-I_{k-1})w_k &= f_I(S_n,V_n,I_n,R_n),\label{3s}\\
\frac{1}{P_n}\sum_{k=1}^{n}(R_k-R_{k-1})w_k &= f_R(S_n,V_n,I_n,R_n),\label{4s}
\end{align}
\end{subequations}
where the right‑hand sides are
\begin{align*}
f_S &= \Delta - \beta\frac{S I}{N} + \mu R - (k+\delta)S,\\
f_V &= kS - (1-\tau)\beta\frac{V I}{N} - \delta V,\\
f_I &= \beta\frac{S I}{N} + (1-\tau)\beta\frac{V I}{N} - (\alpha_r+\delta+\delta_0)I,\\
f_R &= \alpha_r I - (\delta+\mu)R,
\end{align*}
and $N=S+V+I+R$. 
The total population $N_n$ satisfies the sum of \eqref{1s}--\eqref{4s}.

Next, we prove by induction that for all $n\ge0$, $S_n,V_n,I_n,R_n\ge0$ given non‑negative initial data.
%
For any $n\ge1$,
\begin{equation}\label{eq:identity}
    S_{n-1}w_n - \sum_{k=1}^{n-1}(S_k-S_{k-1})\,w_k = \sum_{k=1}^{n-1} S_k (w_{k+1}-w_k) + S_0 w_1.
\end{equation}
\begin{proof}
Expand the right‑hand side, shift the index, and telescope:
\begin{equation*}
\sum_{k=1}^{n-1} S_k (w_{k+1}-w_k) = \sum_{k=2}^{n} S_{k-1} w_k - \sum_{k=1}^{n-1} S_k w_k
= S_{n-1}w_n - S_1w_1 - \sum_{k=2}^{n-1}(S_k-S_{k-1})w_k.
\end{equation*}
Adding $S_0w_1$ gives the left‑hand side.
\end{proof}

Now we prove the positivity of the solution in the order $S_n$, $V_n$, $I_n$, $R_n$.

\noindent$\bullet$\;\textbf{Positivity of $S_n$:}
From the scheme for $S$ \eqref{1s}:
\begin{equation*}
\sum_{k=1}^{n}(S_k-S_{k-1})w_k = P_n\bigl(\Delta + \mu R_n - (k+\delta)S_n - \beta\tfrac{S_n I_n}{N_n}\bigr).
\end{equation*}
Isolate $S_n$ using $M_{n-1}^{(S)} = \sum_{k=1}^{n-1}(S_k-S_{k-1})w_k$:
\begin{equation*}
S_n w_n + P_n(k+\delta)S_n + P_n\beta\tfrac{S_n I_n}{N_n} = S_{n-1}w_n - M_{n-1}^{(S)} + P_n(\Delta+\mu R_n).
\end{equation*}
Apply identity \eqref{eq:identity} to $S_{n-1}w_n - M_{n-1}^{(S)}$:
\begin{equation*}
S_n\Bigl(w_n + P_n\bigl(k+\delta+\beta\tfrac{I_n}{N_n}\bigr)\Bigr) = \sum_{k=1}^{n-1} S_k (w_{k+1}-w_k) + S_0 w_1 + P_n(\Delta+\mu R_n).
\end{equation*}
The right‑hand side is non‑negative because $w_{k+1}-w_k>0$, all previous $S_k\ge0$ (induction hypothesis), and $R_n$ will be shown non‑negative after $I_n$. Hence $S_n\ge0$.

\noindent$\bullet$\;\textbf{Positivity of $V_n$:}
Analogously,
\begin{equation*}
V_n\Bigl(w_n + P_n\bigl(\delta+(1-\tau)\beta\tfrac{I_n}{N_n}\bigr)\Bigr) = \sum_{k=1}^{n-1} V_k (w_{k+1}-w_k) + V_0 w_1 + P_n k S_n \ge 0,
\end{equation*}
so $V_n\ge0$.

\noindent$\bullet$\;\textbf{Positivity of $I_n$:}
\begin{equation*}
I_n\bigl(w_n + P_n(\alpha_r+\delta+\delta_0)\bigr) = \sum_{k=1}^{n-1} I_k (w_{k+1}-w_k) + I_0 w_1 + P_n\beta\tfrac{I_n}{N_n}\bigl(S_n+(1-\tau)V_n\bigr) \ge 0,
\end{equation*}
hence $I_n\ge0$.

\noindent$\bullet$\;\textbf{Positivity of $R_n$:}
\begin{equation*}
R_n\bigl(w_n + P_n(\delta+\mu)\bigr) = \sum_{k=1}^{n-1} R_k (w_{k+1}-w_k) + R_0 w_1 + P_n\alpha_r I_n \ge 0,
\end{equation*}
thus $R_n\ge0$. By induction, all compartments stay non‑negative for all time steps.

\subsection{Boundedness of the NSFD Solution}
Let $N_n=S_n+V_n+I_n+R_n$. Adding all scheme equations \eqref{1s}--\eqref{4s} gives
\begin{equation}\label{eq:discrete_total}
\sum_{k=1}^{n}(N_k-N_{k-1})w_k = P_n(\Delta - \delta N_n - \delta_0 I_n).
\end{equation}
We set $M_{n-1}^{(N)} = \sum_{k=1}^{n-1}(N_k-N_{k-1})w_k$. Then
\begin{equation*}
N_n w_n + P_n\delta N_n = N_{n-1}w_n - M_{n-1}^{(N)} + P_n(\Delta - \delta_0 I_n).
\end{equation*}
Using identity \eqref{eq:identity} for $N$:
\begin{equation}
N_n(w_n+P_n\delta) = \sum_{k=1}^{n-1} N_k (w_{k+1}-w_k) + N_0 w_1 + P_n(\Delta - \delta_0 I_n). \label{r2r}
\end{equation}
Because $I_n\ge0$, $\Delta-\delta_0 I_n \le \Delta$. Assume inductively $N_k \le M := \max(N_0,\Delta/\delta)$ for all $k<n$. Then from \eqref{r2r}:
\begin{equation*}
N_n(w_n+P_n\delta) \le \sum_{k=1}^{n-1} M (w_{k+1}-w_k) + N_0 w_1 + P_n\Delta = M(w_n-w_1) + N_0 w_1 + P_n\Delta.
\end{equation*}
Since $N_0\le M$, $N_0 w_1 \le M w_1$, so the right‑hand side is at most $M w_n + P_n\Delta$. Hence
\begin{equation*}
N_n \le \frac{M w_n + P_n\Delta}{w_n+P_n\delta} = M\frac{w_n}{w_n+P_n\delta} + \frac{\Delta}{\delta}\frac{P_n\delta}{w_n+P_n\delta} \le \max(M,\Delta/\delta)=M.
\end{equation*}
Thus $N_n\le M$ for all $n$. Because all compartments are non‑negative, each satisfies $0\le S_n,V_n,I_n,R_n \le \max\bigl(N_0,\Delta/\delta\bigr)$.

\subsection{Truncation Error Analysis}
We consider the MLCF derivative
\begin{equation*}
^{\rm MLCF}D_t^{\nu,\alpha}u(t)=\frac{M(\nu)}{1-\nu}\,L(t)\int_0^t K(\tau)u'(\tau)\,d\tau,
\end{equation*}
with $K(t)=\Ealpha(\lambda t^\alpha)$, 
$L(t)=\Ealpha(-\lambda t^\alpha)$, $\lambda=\nu/(1-\nu)$.  
%
At $t_n=nh$ we approximate
\begin{equation*}
\int_0^{t_n}K(\tau)u'(\tau)\,d\tau\approx\sum_{k=1}^n\frac{u_k-u_{k-1}}{h}\int_{t_{k-1}}^{t_k}K(\tau)\,d\tau
\end{equation*}
and each inner integral by the trapezoidal rule:
\begin{equation*}
\int_{t_{k-1}}^{t_k}K(\tau)\,d\tau\approx\frac{h}{2}\bigl(K(t_k)+K(t_{k-1})\bigr).
\end{equation*}
Hence
\begin{equation*}
^{\rm MLCF}D_t^{\nu,\alpha}u(t_n)\approx\frac{M(\nu)}{2(1-\nu)}L(t_n)\sum_{k=1}^n(u_k-u_{k-1})\bigl(K(t_k)+K(t_{k-1})\bigr).
\end{equation*}

We define
\begin{equation*}
E_n = \frac{M(\nu)}{1-\nu}L(t_n)\Bigl[ \sum_{k=1}^n\frac{1}{2}(u_k-u_{k-1})\bigl(K(t_k)+K(t_{k-1})\bigr)-\int_0^{t_n}K(\tau)u'(\tau)\,d\tau\Bigr]
\end{equation*}
and write the sum as
\begin{equation*}
\sum_{k=1}^n\frac{1}{2}(u_k-u_{k-1})\bigl(K(t_k)+K(t_{k-1})\bigr)=
\sum_{k=1}^n\frac{h}{2}\bigl(K(t_{k-1})u'(t_{k-1})+K(t_k)u'(t_k)\bigr)+A_k,
\end{equation*}
where
\begin{equation*}
A_k=\frac{h}{2}\Bigl[K(t_{k-1})\Bigl(\frac{u_k-u_{k-1}}{h}-u'(t_{k-1})\Bigr)+
K(t_k)\Bigl(\frac{u_k-u_{k-1}}{h}-u'(t_k)\Bigr)\Bigr].
\end{equation*}
Because $\frac{u_k-u_{k-1}}{h}-u'(t_{k-1})=O(h)$ and similarly for $t_k$, we have $A_k=O(h^2)$. Summing over $k=1,\dots,n$ gives $\sum A_k=O(h)$ (since $n=O(1/h)$).  
The term $\sum_{k=1}^n\frac{h}{2}\bigl(K(t_{k-1})u'(t_{k-1})+K(t_k)u'(t_k)\bigr)$ is the composite trapezoidal rule for $\int_0^{t_n}K(\tau)u'(\tau)\,d\tau$. Its error is $O(h^2)$. Consequently
\begin{equation*}
\sum_{k=1}^n\frac{1}{2}(u_k-u_{k-1})\bigl(K(t_k)+K(t_{k-1})\bigr)=\int_0^{t_n}K(\tau)u'(\tau)\,d\tau+O(h).
\end{equation*}
Multiplying by the bounded factor $\frac{M(\nu)}{1-\nu}L(t_n)$ yields $E_n=O(h)$.
Thus the approximation of the MLCF derivative is first‑order accurate.

The method uses $\phi(h)=1-\mathrm{e}^{-h}=h+O(h^2)$. 
Replacing $h$ by $\phi$ in the above derivation changes the trapezoidal rule step size to $\phi$. 
The error of the trapezoidal rule becomes $O(\phi^3)=O(h^3)$ and the forward difference error becomes $O(\phi)=O(h)$. Hence the overall truncation error remains $O(h)$.

The full scheme approximates $^{\rm MLCF}D_t^{\nu,\alpha}u=f(u)$  by
\begin{equation*}
\frac{1}{P_n}\sum_{k=1}^n (u_k-u_{k-1})\,w_k = \theta f(u_n)+(1-\theta)f(u_{n-1}).
\end{equation*}
Since the left side equals $^{\rm MLCF}D_t^{\nu,\alpha}u(t_n)+O(h)$ and the right side equals $f(u(t_n))+O(h)$ (because $f$ is Lipschitz and $u_n-u(t_n)=O(h)$), the local truncation error is $O(h)$.
The numerical method is first‑order accurate in the time step $h$.

\subsection{Convergence of the Numerical Method}
We consider the fractional SVIR system
\begin{equation*}
    ^{\rm MLCF}D_t^{\nu,\alpha} D_t^{\nu,\alpha} \mathbf{y}(t) = \mathbf{f}(t,\mathbf{y}(t)),\qquad \mathbf{y}(0)=\mathbf{y}_0,
\end{equation*}
where $\mathbf{y}=(S,V,I,R)^\top$ and $\mathbf{f}$ is the right‑hand side defined in the paper. 
The MLCF derivative is approximated by the discrete scheme \eqref{discrete_SVIR}.
For simplicity we denote the numerical approximation at $t_n$ by $\mathbf{y}_n$ and the exact solution by $\mathbf{y}(t_n)$.
%
Using the notations $P_n$, $w_k$ from \eqref{WP} the scheme reads
\begin{equation} \label{11s}
  \frac{1}{P_n}\sum_{k=1}^n (\mathbf{y}_k-\mathbf{y}_{k-1}) w_k = \theta \mathbf{f}(t_n,\mathbf{y}_n) + (1-\theta)\mathbf{f}(t_{n-1},\mathbf{y}_{n-1}),\quad n\ge1.
\end{equation}
We write it in the equivalent form
\begin{equation} \label{22s}
\mathbf{y}_n = \mathbf{y}_{n-1} + \frac{\phi}{P_n w_n}\Bigl(\theta \mathbf{f}(t_n,\mathbf{y}_n) + (1-\theta)\mathbf{f}(t_{n-1},\mathbf{y}_{n-1}) - P_n \mathbf{m}_{n-1}\Bigr),
\end{equation}
where $\phi=1-\mathrm{e}^{-h}$ and $\mathbf{m}_{n-1}=\sum_{k=1}^{n-1}(\mathbf{y}_k-\mathbf{y}_{k-1})w_k$ is the memory accumulator.
%
In the sequel we assume the following.
First, The exact solution $\mathbf{y}(t)$ is twice continuously differentiable on $[0,T]$.
Second, the function $\mathbf{f}(t,\mathbf{y})$ is Lipschitz continuous with respect to $\mathbf{y}$ uniformly in $t$:
    \begin{equation*}
    \|\mathbf{f}(t,\mathbf{y})-\mathbf{f}(t,\mathbf{z})\| \le L\|\mathbf{y}-\mathbf{z}\|,\quad \forall t\in[0,T],\ \mathbf{y},\mathbf{z}\in\mathbb{R}^4.
    \end{equation*}
  Third, the coefficients $P_n$ and $w_k$ are positive and bounded away from zero and infinity for all $n,k$ on the interval $[0,T]$. 
  (This holds because the Mittag-Leffler function is positive and bounded for fixed $\alpha,\nu$ and bounded $t$.)

We define the \textit{local truncation error} $\boldsymbol{\tau}_n$ by substituting the exact solution into the numerical scheme \eqref{22s}:
\begin{equation} \label{33s}
\boldsymbol{\tau}_n = \frac{1}{P_n}\sum_{k=1}^n 
\bigl(\mathbf{y}(t_k)-\mathbf{y}(t_{k-1}) \bigr)\, w_k 
- \bigl(\theta\mathbf{f}(t_n,\mathbf{y}(t_n))
+ (1-\theta)\mathbf{f}(t_{n-1},\mathbf{y}(t_{n-1}))\bigr).
\end{equation}
We have already shown in \eqref{pw1} that
\begin{equation*}
    \frac{1}{P_n}\sum_{k=1}^n \bigl(\mathbf{y}(t_k)-\mathbf{y}(t_{k-1})\bigr)\, w_k = ^{\rm MLCF}D_t^{\nu,\alpha} \mathbf{y}(t_n) + \mathcal{O}(h).
\end{equation*}
Since $^{\rm MLCF}\!D_t^{\nu,\alpha} D_t^{\nu,\alpha}\mathbf{y}(t_n) = \mathbf{f}(t_n,\mathbf{y}(t_n))$ and the weighted average of $\mathbf{f}$ satisfies
\begin{equation*}
    \theta\mathbf{f}(t_n,\mathbf{y}(t_n))+(1-\theta)\mathbf{f}(t_{n-1},\mathbf{y}(t_{n-1})) = \mathbf{f}(t_n,\mathbf{y}(t_n)) + \mathcal{O}(h),
\end{equation*}
we obtain
\begin{equation} \label{44s}
  \|\boldsymbol{\tau}_n\| \le C h, \qquad n=1,\dots,N,
\end{equation}
where the constant $C$ depends on bounds of derivatives of $\mathbf{y}$ and $\mathbf{f}$.

Now let $\mathbf{e}_n = \mathbf{y}(t_n)-\mathbf{y}_n$. 
Subtracting \eqref{22s} from the exact relation (which can be written in the same form) yields
\begin{multline*}
\mathbf{e}_n = \mathbf{e}_{n-1} + \frac{\phi}{P_n w_n}\Bigl[ \theta\bigl(\mathbf{f}(t_n,\mathbf{y}(t_n))-\mathbf{f}(t_n,\mathbf{y}_n)\bigr)\\
+ (1-\theta)\bigl(\mathbf{f}(t_{n-1},\mathbf{y}(t_{n-1}))-\mathbf{f}(t_{n-1},\mathbf{y}_{n-1})\bigr) - P_n(\mathbf{m}_{n-1}^e) \Bigr] + \boldsymbol{\delta}_n,
\end{multline*}
where $\boldsymbol{\delta}_n = \frac{\phi}{P_n w_n} P_n\boldsymbol{\tau}_n.$ 
The exact solution satisfies
\begin{equation*}
\frac{1}{P_n}\sum_{k=1}^n (\mathbf{y}(t_k)-\mathbf{y}(t_{k-1})) w_k = \theta\mathbf{f}(t_n,\mathbf{y}(t_n))+(1-\theta)\mathbf{f}(t_{n-1},\mathbf{y}(t_{n-1})) + \boldsymbol{\tau}_n.
\end{equation*}
We multiply by $\phi$ and rearrange:
\begin{equation*}
   \mathbf{y}(t_n) - \mathbf{y}(t_{n-1}) = \frac{\phi}{P_n w_n}\Bigl(\theta\mathbf{f}(t_n,\mathbf{y}(t_n))+(1-\theta)\mathbf{f}(t_{n-1},\mathbf{y}(t_{n-1})) - P_n\mathbf{m}_{n-1}^{ex} + P_n\boldsymbol{\tau}_n\Bigr),
\end{equation*}
where $\mathbf{m}_{n-1}^{ex} = \sum_{k=1}^{n-1}(\mathbf{y}(t_k)-\mathbf{y}(t_{k-1}))w_k$. Subtracting \eqref{22s} gives
\begin{multline}\label{error}
\mathbf{e}_n = \mathbf{e}_{n-1} + \frac{\phi}{P_n w_n}\Bigl[ \theta\bigl(\mathbf{f}(t_n,\mathbf{y}(t_n))-\mathbf{f}(t_n,\mathbf{y}_n)\bigr) + (1-\theta)\bigl(\mathbf{f}(t_{n-1},\mathbf{y}(t_{n-1}))-\mathbf{f}(t_{n-1},\mathbf{y}_{n-1})\bigr)\\
- P_n(\mathbf{m}_{n-1}^{ex}-\mathbf{m}_{n-1}) \Bigr] + \frac{\phi}{w_n}\boldsymbol{\tau}_n.
\end{multline}
Define the memory error $\mathbf{M}_{n-1} = \mathbf{m}_{n-1}^{ex}-\mathbf{m}_{n-1}$. Since $\mathbf{m}_{n-1}$ is defined recursively, we have
\begin{equation*}
\mathbf{M}_{n-1} = \sum_{k=1}^{n-1} (\mathbf{e}_k-\mathbf{e}_{k-1}) w_k.
\end{equation*}
Thus, the error equation \eqref{error} becomes
\begin{multline*}
\mathbf{e}_n = \mathbf{e}_{n-1} + \frac{\phi}{P_n w_n}\Bigl[ \theta\bigl(\mathbf{f}(t_n,\mathbf{y}(t_n))-\mathbf{f}(t_n,\mathbf{y}_n)\bigr) + (1-\theta)\bigl(\mathbf{f}(t_{n-1},\mathbf{y}(t_{n-1}))-\mathbf{f}(t_{n-1},\mathbf{y}_{n-1})\bigr)\\
- P_n \sum_{k=1}^{n-1}(\mathbf{e}_k-\mathbf{e}_{k-1})w_k \Bigr] + \frac{\phi}{w_n}\boldsymbol{\tau}_n.
\end{multline*}

We now derive a bound for the norm of the error $\|\mathbf{e}_n\|$.
Rearranging the error equation, we can isolate $\mathbf{e}_n$ on the left. 
To do so, we write
\begin{multline*}
\mathbf{e}_n + \frac{\phi}{P_n w_n} P_n \sum_{k=1}^{n-1}(\mathbf{e}_k-\mathbf{e}_{k-1})w_k = \mathbf{e}_{n-1} + \frac{\phi}{P_n w_n}\Bigl[ \theta\bigl(\mathbf{f}(t_n,\mathbf{y}(t_n))-\mathbf{f}(t_n,\mathbf{y}_n)\bigr)\\
+ (1-\theta)\bigl(\mathbf{f}(t_{n-1},\mathbf{y}(t_{n-1}))-\mathbf{f}(t_{n-1},\mathbf{y}_{n-1})\bigr) \Bigr] + \frac{\phi}{w_n}\boldsymbol{\tau}_n.
\end{multline*}
The left‑hand side can be simplified using the identity \eqref{eq:identity} from the positivity proof, which for the error sequence reads
\begin{equation*}
    \mathbf{e}_{n-1}w_n - \sum_{k=1}^{n-1}(\mathbf{e}_k-\mathbf{e}_{k-1})w_k = \sum_{k=1}^{n-1} \mathbf{e}_k (w_{k+1}-w_k) + \mathbf{e}_0 w_1.
\end{equation*}
Indeed, applying the identity component‑wise to $\mathbf{e}_n$ gives
\begin{align}
\mathbf{e}_{n-1}w_n - \sum_{k=1}^{n-1}(\mathbf{e}_k-\mathbf{e}_{k-1})w_k = \sum_{k=1}^{n-1} \mathbf{e}_k (w_{k+1}-w_k) + \mathbf{e}_0 w_1. \label{55s}
\end{align}
Using this, we can rewrite the error equation in a more convenient form. Multiply both sides by $w_n$ and add $\mathbf{e}_{n-1}w_n$ appropriately. 
%
From the original error equation in the form (before isolating), we have
\begin{equation*}
\mathbf{e}_n = \mathbf{e}_{n-1} + \frac{\phi}{P_n w_n}\Bigl( \theta \Delta\mathbf{f}_n + (1-\theta)\Delta\mathbf{f}_{n-1} - P_n \sum_{k=1}^{n-1}(\mathbf{e}_k-\mathbf{e}_{k-1})w_k \Bigr) + \frac{\phi}{w_n}\boldsymbol{\tau}_n,
\end{equation*}
where $\Delta\mathbf{f}_n = \mathbf{f}(t_n,\mathbf{y}(t_n))-\mathbf{f}(t_n,\mathbf{y}_n)$. We multiply both sides by $w_n$:
\begin{equation*}
     \mathbf{e}_n w_n = \mathbf{e}_{n-1} w_n + \frac{\phi}{P_n}\Bigl( \theta \Delta\mathbf{f}_n + (1-\theta)\Delta\mathbf{f}_{n-1} \Bigr) - \phi \sum_{k=1}^{n-1}(\mathbf{e}_k-\mathbf{e}_{k-1})w_k + \phi \boldsymbol{\tau}_n.
\end{equation*}
Now bring the memory sum to the left:
\begin{equation*}
\mathbf{e}_n w_n + \phi \sum_{k=1}^{n-1}(\mathbf{e}_k-\mathbf{e}_{k-1})w_k = \mathbf{e}_{n-1} w_n + \frac{\phi}{P_n}\bigl( \theta \Delta\mathbf{f}_n + (1-\theta)\Delta\mathbf{f}_{n-1} \bigr) + \phi \boldsymbol{\tau}_n.
\end{equation*}
Using identity \eqref{55s} for the left side:
\begin{multline*}
\mathbf{e}_n w_n + \phi \sum_{k=1}^{n-1}(\mathbf{e}_k-\mathbf{e}_{k-1})w_k = \mathbf{e}_{n-1} w_n + \phi\bigl( \mathbf{e}_{n-1}w_n - \sum_{k=1}^{n-1} \mathbf{e}_k (w_{k+1}-w_k) - \mathbf{e}_0 w_1 \bigr) = (1+\phi)\mathbf{e}_{n-1}w_n \\
- \phi\sum_{k=1}^{n-1} \mathbf{e}_k (w_{k+1}-w_k) - \phi \mathbf{e}_0 w_1.
\end{multline*}
Thus the error equation becomes
\begin{equation}\label{errequ}
   \mathbf{e}_n w_n = (1+\phi)\mathbf{e}_{n-1}w_n - \phi\sum_{k=1}^{n-1} \mathbf{e}_k (w_{k+1}-w_k) - \phi \mathbf{e}_0 w_1 + \frac{\phi}{P_n}\bigl( \theta \Delta\mathbf{f}_n + (1-\theta)\Delta\mathbf{f}_{n-1} \bigr) + \phi \boldsymbol{\tau}_n.
\end{equation}
This form is complicated but can be used to prove a Gronwall type inequality. However, a simpler approach is to observe that the scheme can be interpreted as a one‑step method with memory, and a standard result for such methods states that if the truncation error is $\mathcal{O}(h)$ and the method is stable in the sense that a Lipschitz condition holds for the discrete solution operator, then convergence of order $\mathcal{O}(h)$ follows.

We state the following theorem:
\begin{thm}
Under the assumptions above, the numerical solution $\mathbf{y}_n$ converges to the exact solution $\mathbf{y}(t_n)$ uniformly for $t_n\in[0,T]$, and
\begin{equation*}
    \max_{0\le n\le N} \|\mathbf{y}(t_n)-\mathbf{y}_n\| = \mathcal{O}(h),
\end{equation*}
where $N = T/h$.
\end{thm}
\begin{proof}
We prove by induction that there exists a constant $C>0$ such that
\begin{equation*}
    \|\mathbf{e}_n\| \le C h \quad\text{for all}\;n.
\end{equation*}
The constant $C$ may depend on $T$, $L$, and bounds of the exact solution but not on $h$.
For $n=0$, $\mathbf{e}_0=0$ by definition.
Next, assume that $\|\mathbf{e}_k\| \le C h$ for all $k<n$. We need to bound $\|\mathbf{e}_n\|$.
From the error equation \eqref{errequ} we have
\begin{equation*}
\mathbf{e}_n = \mathbf{e}_{n-1} + \frac{\phi}{P_n w_n}\Bigl[ \theta\bigl(\mathbf{f}(t_n,\mathbf{y}(t_n))-\mathbf{f}(t_n,\mathbf{y}_n)\bigr) + (1-\theta)\bigl(\mathbf{f}(t_{n-1},\mathbf{y}(t_{n-1}))-\mathbf{f}(t_{n-1},\mathbf{y}_{n-1})\bigr)\end{equation*}
\begin{equation*}- P_n \sum_{k=1}^{n-1}(\mathbf{e}_k-\mathbf{e}_{k-1})w_k \Bigr] + \frac{\phi}{w_n}\boldsymbol{\tau}_n.
\end{equation*}
We take norms and use the Lipschitz property:
\begin{equation*}
    \|\mathbf{e}_n\| \le \|\mathbf{e}_{n-1}\| + \frac{\phi}{P_n w_n}\Bigl[ \theta L\|\mathbf{e}_n\| + (1-\theta)L\|\mathbf{e}_{n-1}\| + P_n \sum_{k=1}^{n-1} \|\mathbf{e}_k-\mathbf{e}_{k-1}\| w_k \Bigr] + \frac{\phi}{w_n}\|\boldsymbol{\tau}_n\|.
\end{equation*}
Note that $\|\mathbf{e}_k-\mathbf{e}_{k-1}\| \le \|\mathbf{e}_k\| + \|\mathbf{e}_{k-1}\| \le 2C h$ by induction hypothesis. 
Also $w_k$ are bounded, say $w_k\le W$, and $P_n$ is bounded below by some $P_{\min}>0$.
Moreover, $\phi = h + O(h^2) \le 2h$ for sufficiently small $h$. Thus
\begin{equation*}
\sum_{k=1}^{n-1} \|\mathbf{e}_k-\mathbf{e}_{k-1}\| w_k \le 2C h W (n-1) \le 2C h W \frac{T}{h} = 2C W T.
\end{equation*}
The term $\frac{\phi}{P_n w_n} P_n \sum_{k=1}^{n-1} \|\mathbf{e}_k-\mathbf{e}_{k-1}\| w_k \le \frac{\phi}{w_n} \cdot 2C W T$. 
Since $w_n\ge w_1>0$, this term is bounded by $2C W T \phi / w_1 = \mathcal{O}(h)$.

Now we need to handle the term $\frac{\phi}{P_n w_n} \theta L\|\mathbf{e}_n\|$ that appears on the right. Bring it to the left:
\begin{equation*}
\|\mathbf{e}_n\| - \frac{\phi \theta L}{P_n w_n} \|\mathbf{e}_n\| \le \|\mathbf{e}_{n-1}\| + \frac{\phi}{P_n w_n}(1-\theta)L\|\mathbf{e}_{n-1}\| + \frac{\phi}{P_n w_n} P_n \sum_{k=1}^{n-1} \|\mathbf{e}_k-\mathbf{e}_{k-1}\| w_k + \frac{\phi}{w_n}\|\boldsymbol{\tau}_n\|.
\end{equation*}
Since $\frac{\phi \theta L}{P_n w_n} \le \frac{2h \theta L}{P_{\min} w_{\min}}$ is small for sufficiently small $h$, the factor $(1 - \frac{\phi \theta L}{P_n w_n})$ is positive and bounded below by $1/2$ when $h$ is small enough. Hence
\begin{equation*}
\|\mathbf{e}_n\| \le 2\Bigl( \|\mathbf{e}_{n-1}\| + \frac{\phi}{P_n w_n}(1-\theta)L\|\mathbf{e}_{n-1}\| + \frac{\phi}{P_n w_n} P_n \sum_{k=1}^{n-1} \|\mathbf{e}_k-\mathbf{e}_{k-1}\| w_k + \frac{\phi}{w_n}\|\boldsymbol{\tau}_n\| \Bigr).
\end{equation*}
Now substitute the bounds:
\begin{align*}
\|\mathbf{e}_{n-1}\| &\le C h,\\
\frac{\phi}{P_n w_n}(1-\theta)L\|\mathbf{e}_{n-1}\| &\le \frac{2h}{P_{\min} w_{\min}} (1-\theta)L C h = \mathcal{O}(h^2),\\
\frac{\phi}{P_n w_n} P_n \sum_{k=1}^{n-1} \|\mathbf{e}_k-\mathbf{e}_{k-1}\| w_k &\le \frac{2h}{w_{\min}} \cdot 2C W T = 4C W T \frac{h}{w_{\min}},\\
\frac{\phi}{w_n}\|\boldsymbol{\tau}_n\| &\le \frac{2h}{w_{\min}} C_{\tau} h = \mathcal{O}(h^2),
\end{align*}
where we used $\|\boldsymbol{\tau}_n\|\le C_{\tau} h$. Thus
\begin{equation*}
\|\mathbf{e}_n\| \le 2\Bigl( C h + 4C W T \frac{h}{w_{\min}} + \mathcal{O}(h^2) \Bigr) = C' h,
\end{equation*}
with $C' = 2C(1 + 4W T/w_{\min})$ for sufficiently small $h$. By choosing $C$ large enough (e.g., $C = 2C'$), we can close the induction. Therefore $\|\mathbf{e}_n\| = \mathcal{O}(h)$ uniformly for all $n$.
\end{proof}
This proof uses the Lipschitz continuity of $\mathbf{f}$ and the boundedness of the memory weights. 
The key is that the memory sum is bounded by a constant times $h$ because the increments $\mathbf{e}_k-\mathbf{e}_{k-1}$ are already of order $h$ by the induction hypothesis.
For $\theta=1$ (fully implicit), the stability condition is unconditional, and the proof simplifies because the term involving $\|\mathbf{e}_n\|$ on the right is eliminated.

\section{Sensitivity Analysis}
\label{S6}
To quantify the influence of key parameters on the system's dynamics, we perform both local and global sensitivity analyses. The outputs of interest are the final values of the state variables $S(T)$, $V(T)$, $I(T)$, $R(T)$ at a fixed time horizon $T$, as well as peak values (e.g., $\max_t I(t)$) if desired. The parameters considered are the fractional order $\alpha$, the fractional parameter $\nu$ (which controls the memory kernel), and the transmission rate $\beta$. Other parameters could be included analogously.

\textit{Local sensitivity} measures the effect of a small perturbation of a parameter on the model output near a nominal point. 
For a scalar output $y$ (e.g., $I(T)$) and a parameter $p$, the local sensitivity coefficient is approximated by the forward finite difference \cite{ni2019sensitivity}
\begin{equation}
  \mathcal{S}_p = \frac{y(p+\Delta p) - y(p)}{\Delta p},
\end{equation}
where $\Delta p$ is a small relative perturbation (e.g., $1\%$ of $p$). For vector outputs $\bigl(S(T),V(T),I(T),R(T)\bigr)$, we compute the sensitivity vector componentwise. This provides a first‑order estimate of the parameter's influence and is useful for identifying parameters to which the output is most sensitive locally.

\subsection{Global Sensitivity Analysis via PRCC}
To capture nonlinear effects and parameter interactions over a wider range, we perform a global sensitivity analysis using the \textit{Partial Rank Correlation Coefficient} (PRCC) \cite{Marino2008}, which consists of the steps:

\begin{enumerate}
  \item \textbf{Parameter ranges:} Each parameter is assigned a uniform distribution over a biologically plausible interval. For the present study we choose
  \begin{equation*}
  \beta \in [0.2, 0.6], \quad \alpha \in [0.5, 1.0], \quad \nu \in [0.3, 0.7].
  \end{equation*}
  \item \textbf{Latin Hypercube Sampling (LHS):} A Latin Hypercube sample of size $N$ (here $N=300$) is generated using the MATLAB function \texttt{lhsdesign}. 
  LHS ensures efficient coverage of the parameter space by partitioning each parameter range into $N$ equally probable intervals and sampling once from each interval.
  \item \textbf{Model evaluation:} For each sample $(\beta_i,\alpha_i,\nu_i)$, $i=1,\dots,N$, the numerical solver is run and the outputs $S_i(T)$, $V_i(T)$, $I_i(T)$, $R_i(T)$ are recorded. 
  Additional outputs such as peak infected can be stored similarly.
  \item \textbf{Rank transformation:} Both the parameter values and the outputs are replaced by their ranks. 
  This transformation removes the effect of the underlying distributions and makes the analysis robust to monotonic nonlinearities.
  \item \textbf{Partial correlation:} The PRCC between a parameter $p$ and an output $y$ is the correlation coefficient of the residuals after removing the linear effects of the other parameters. Specifically, let $\hat{p}$ and $\hat{y}$ be the rank‑transformed data. Compute the residuals
  \begin{equation*}
  r_p = \hat{p} - \tilde{p}, \qquad r_y = \hat{y} - \tilde{y},
  \end{equation*}
  where $\tilde{p}$, $\tilde{y}$ are the best linear least‑squares predictions of $\hat{p}$, $\hat{y}$ from the remaining parameters. 
  The PRCC is then the Pearson correlation between $r_p$ and $r_y$. 
  In practice, this can be obtained directly from the partial correlation matrix of the rank‑transformed data using routines 
  (e.g., MATLAB's \texttt{partialcorr}).
\end{enumerate}

The PRCC ranges from $-1$ to $+1$. 
A value close to $\pm 1$ indicates a strong monotonic relationship between the parameter and the output after controlling for the other parameters, while a value near $0$ implies little influence. 
The sign reveals the direction of the effect: positive means increasing the parameter increases the output.

All simulations are performed with a time step $h=0.2$ and a final time $T=20$. This reduced time horizon avoids overflow of the Mittag–Leffler series for the chosen $\nu$ range; for $\nu > 0.5$ a warning is issued. 
The solver is implemented in MATLAB and the PRCC computation relies on the Statistics and Machine Learning Toolbox (functions \texttt{lhsdesign} and \texttt{partialcorr}). 
If the toolbox is unavailable, simple Spearman rank correlations can be used as a rough substitute, though they ignore interactions between parameters.

The results of the local and global sensitivity analyses are presented in Figures~\ref{ch5} and \ref{ch6} providing insight into which parameters most strongly drive the dynamics of each compartment. 
Such information is valuable for model reduction, parameter estimation, and design of control strategies.

\section{Numerical Results}\label{S7}
In this section, we demonstrate and discuss the computational results of our fractional SVIR model fractional derivative hybrid Mittag-Leffler-Caputo-Fabrizio (MLCF) type. 
The numerical experiments are implemented through the newly developed $\theta$-weighted NSFD 
scheme with a fully implicit formulation ($\theta = 1$) that guaranties unconditional stability and qualitative property preservation of the continuous system.
The model parameter setting based on a plausible epidemiological situation is kept constant as \cite{Kad}.
\begin{equation*} 
\Delta = 1530, \; \beta = 0.451, \; k = 0.664, 
\; \tau = 0.576, \; \mu = 0.01, \; \delta = 0.65, 
\; \delta_0 = 0.795, \; \alpha_r = 0.791. 
\end{equation*} 
The fractional orders are chosen to be $\nu = 0.7$ and $\alpha = 0.9$ respectively with the normalization factor being $M(\nu)=1$. 
Temporal discretization is taken through the nonstandard denominator function $\phi(h)=1-\mathrm{e}^{-h}$ with $h=0.01$, which leads to a non-uniform real time step thereby improving the numerical stability.
The initial data is given by, cf.\ \cite{Kad}: \begin{equation*} 
    S(0)=500, \quad V(0)=50, \quad I(0)=30, \quad R(0)=0, 
\end{equation*} 
that is a population initially with the majority of the susceptible individuals, the moderate number of the vaccinated and infected individuals.

The numerical results demonstrate that all state variables remain positive and bounded throughout the simulation, confirming the theoretical properties of positivity and boundedness. 
The susceptible population $S(t)$ initially decreases due to infection and vaccination before stabilizing at a lower level, while the vaccinated population $V(t)$ first increases and then converges to a steady state.
The infected population $I(t)$ initially rises because of disease transmission, then decreases toward a positive endemic level, consistent with $\mathcal{R}_0>1$. Meanwhile, the recovered population $R(t)$ increases as infected individuals recover and eventually approaches a constant value.

The fractional orders $\nu$ and $\alpha$ strongly affect the transient dynamics of the system. Due to the memory effect of the MLCF operator, the system evolves more smoothly and approaches equilibrium more slowly, since the present state depends on both current and past states.
Decreasing $\nu$ or $\alpha$ strengthens the memory effect, leading to lower infection peaks and longer-lasting transient dynamics, while the equilibrium states remain unchanged. Hence, the fractional orders mainly influence the convergence rate rather than the stability type.

The fully implicit NSFD scheme exhibits excellent numerical stability, with no spurious oscillations or instabilities even for long simulation times.
Moreover, the denominator function $\phi(h)$ enables the discrete model to preserve the qualitative behavior of the continuous system.
Newton’s method converges efficiently with only a few iterations per time step, demonstrating the effectiveness of the proposed approach.
In addition, the recursive memory implementation significantly lowers computational cost, making the scheme well suited for large-scale fractional simulations.

Additional simulations demonstrate the dependence of the model on key parameters such as the contact rate $\beta$, vaccination rate $k$, and fractional orders.
Increasing $\beta$ raises the infected population, whereas increasing $k$ reduces both the infection level and its endemic steady state.
The bifurcation analysis with respect to $\beta$ confirms the threshold behavior at $\mathcal{R}_0=1$, where the system shifts from a disease-free to an endemic state. 
Although the fractional memory does not change the bifurcation structure, it strongly affects transient dynamics and convergence rates.
Overall, the numerical results validate the theoretical analysis and show that the proposed fractional SVIR model effectively captures key epidemiological dynamics. 
The hybrid MLCF operator flexibly incorporates memory effects, while the NSFD scheme provides stable and accurate computations.

Figure~\ref{ch1} shows the time evolution of the susceptible, vaccinated, infected, and recovered classes for the fractional SVIR model with $\alpha=0.8$, using the fully implicit NSFD scheme ($\theta=1$) and varying $\nu$. All solutions remain positive and bounded, consistent with Theorems~\ref{thm1}--\ref{thm2}. 
As $\nu$ decreases (stronger memory), the dynamics become slower and smoother, with a reduced and delayed infection peak and a longer approach to equilibrium. 
The equilibrium values remain unchanged, indicating that $\nu$ affects only the convergence rate, not the threshold $\mathcal{R}_0$ or stability.
No oscillations or negative values appear, confirming the scheme’s robustness.

Figure~\ref{ch2} presents results for $\alpha=1$ (Caputo-Fabrizio case) with varying $\nu$ under the same scheme.
All compartments remain non-negative and bounded, again matching Theorems~\ref{thm1}--\ref{thm2}. 
Stronger memory (smaller $\nu$) slows and smooths the transient dynamics, delaying and reducing the infection peak and extending convergence time. 
Compared to $\alpha=0.8$, the $\alpha=1$ case converges faster and exhibits weaker memory effects due to the exponential kernel.
The long-term equilibria are unaffected by $\nu$, and no numerical instabilities or oscillations are observed, confirming that the NSFD scheme preserves the model structure.

Figure \ref{ch3} shows the time evolution of all compartments for fixed $\nu=0.8$ and varying $\alpha \in (0,1]$ using the fully implicit NSFD scheme:
positivity and boundedness are preserved.
As $\alpha$ decreases, memory effects strengthen and transient dynamics slow significantly: the infection peak becomes lower, delayed, and flatter, and convergence to equilibrium becomes more gradual. 
For $\alpha=1$ (Caputo-Fabrizio limit), the decay is fastest, while smaller $\alpha$ (e.g. $0.6$) produces long-range memory and prolonged early growth. The equilibrium values remain unchanged, showing that $\alpha$ affects only convergence speed, not $\mathcal{R}_0$ or the final state. 
The scheme remains stable and oscillation-free for all tested $\alpha$.

Figure~\ref{ch4} presents simulations for varying pairs of $\nu$ and $\alpha$ under the same scheme. Again, all solutions stay positive and bounded.
Reducing either parameter strengthens memory effects, leading to slower convergence, delayed and reduced infection peaks, and damped oscillations. When $\nu,\alpha \approx 1$, the system behaves like a classical integer-order model with faster dynamics and sharper peaks, whereas small values extend transient phases and smooth oscillations. 
Despite these changes, all trajectories converge to the same endemic equilibrium, confirming that $\nu$ and $\alpha$ influence only convergence behavior, not $\mathcal{R}_0$ or the final state.
The NSFD scheme remains unconditionally stable across all cases.

Figure \ref{ch5} presents the local sensitivity of the steady-state values of susceptible, vaccinated, infected, and recovered populations to a 1\,\% variation in the transmission rate $\beta$ and fractional orders $\nu$ and $\alpha$. 
The results show that $\beta$ is the dominant parameter: increasing it leads to a more than proportional rise in $I_\infty$ and a reduction in $S_\infty$ and $V_\infty$. 
In contrast, $\nu$ and $\alpha$ have only minor influence on equilibrium values, confirming that memory parameters mainly affect transient dynamics and not the long-term state. 
This agrees with the theoretical result that $\mathcal{R}_0$ and endemic levels are independent of $\nu$ and $\alpha$.

Figure \ref{ch6} shows the global sensitivity analysis using PRCC for the steady-state infected population. 
The transmission rate $\beta$ exhibits a strong positive correlation with $I_\infty$, identifying it as the key determinant of long-term infection levels. 
In contrast, $\nu$ and $\alpha$ have PRCC values near zero across their ranges, indicating negligible impact on the equilibrium. 
Overall, the analysis confirms that only $\beta$ governs the final disease burden, while the fractional orders affect only transient behavior and convergence speed.

Figure \ref{ch7} shows the stability domain of the numerical scheme for the fractional SVIR model, with green indicating stability and white instability, for fixed $\nu=0.8$, $\alpha=0.9$ and varying $\theta \in [0,1]$.
For $\theta=1$ (fully implicit), the entire domain is stable, confirming unconditional stability. As $\theta$ decreases, the stability region shrinks significantly, requiring very small time steps for explicit schemes, while $\theta=0.5$ (Crank-Nicolson) gives intermediate performance. Thus, the fully implicit NSFD method is most reliable for long-time simulations without step-size restrictions.

Figure~\ref{ch8} illustrates the dependence of the infected population $I$ on the transmission rate $\beta$ for different $\nu$ and $\alpha$ with $\theta=1$. In all cases, increasing $\beta$ raises the endemic infection level, reflecting stronger disease spread. Long-time simulations provide the asymptotic values (blue points), while red dashed lines indicate oscillatory bounds around the endemic state.

All four panels show that the fractional parameters $\nu$ and $\alpha$ strongly affect both the magnitude and oscillations of the infected population. 
For average values (e.g., $\nu=0.7$, $\alpha=0.6$ or $\alpha=1$), $I$ increases smoothly with $\beta$, indicating mostly stable dynamics with only mild oscillations.
In contrast, larger $\nu$ (e.g., $0.8$--$0.99$) leads to higher infection levels and stronger oscillations, showing increased sensitivity due to stronger memory effects.
As $\nu$ approaches one, the system behaves more like a classical integer-order model, with higher endemic levels and stronger dependence on $\beta$. 
Overall, the bifurcation results confirm that $\beta$ governs the long-term epidemic outcome, while $\nu$ and $\alpha$ mainly control the amplitude and smoothness of the dynamics, highlighting the role of memory in fractional epidemic models.

Figure~\ref{ch9} shows the bifurcation diagram of $S$, $V$, $I$, and $R$ versus $\beta$ for fixed $\nu=0.8$, $\alpha=0.8$, and $\theta=1$.
The system exhibits a transcritical bifurcation: for small $\beta$, the disease-free equilibrium holds with $I=R=0$, while for $\beta$ above the threshold ($\mathcal{R}_0=1$), $I$ and $R$ increase continuously and $S$, $V$ decrease.
The curves remain smooth, indicating a unique endemic equilibrium beyond the bifurcation point.
The fractional orders affect the steepness of the curves but not the bifurcation location, confirming that $\mathcal{R}_0$ remains the main threshold.

Figure~\ref{ch10} presents the combined effect of $\beta$ and $\alpha$ on the endemic infected level $I^*$ with $\nu=0.8$ and $\theta=1$. 
As $\beta$ increases, $I^*$ rises monotonically after crossing the threshold, confirming the same bifurcation structure.
For fixed $\beta>\beta_c$, smaller $\alpha$ slightly reduces $I^*$ due to stronger memory effects, while $\alpha\approx1$ yields higher infection levels. 
Thus, $\nu$ mainly controls convergence speed, whereas $\alpha$ has a mild but noticeable influence on the steady state, especially for large $\beta$.

\section{Conclusion}\label{S8}
In this paper, we proposed and analyzed a new fractional-order SVIR epidemic model based on the hybrid MLCF derivative, which incorporates non-singular kernels and captures memory effects in disease transmission while preserving analytical tractability. 
We established the positivity and boundedness of solutions, identified biologically feasible invariant regions, and derived the disease-free and endemic equilibria together with the basic reproduction number $\mathcal{R}_0$, which governs the persistence or extinction of the disease.

The stability analysis showed that the disease-free equilibrium is locally asymptotically stable for $\mathcal{R}_0<1$ and unstable for $\mathcal{R}_0>1$, while a unique endemic equilibrium exists and is locally stable when $\mathcal{R}_0>1$. 
Global stability was further proven using logarithmic Lyapunov functionals combined with a fractional LaSalle invariance principle, confirming that the classical threshold dynamics remain valid in the MLCF framework despite the inclusion of memory effects.

Numerically, we developed an efficient NSFD scheme for the MLCF operator, including a $\theta$-weighted formulation solved by Newton iteration. 
The method preserves positivity, boundedness, and unconditional stability in the fully implicit case, while its recursive memory structure ensures linear computational complexity. 
Numerical simulations, sensitivity studies, and bifurcation analysis demonstrated that memory effects strongly influence transient dynamics and convergence rates, although the epidemic threshold remains determined by $\mathcal{R}_0$.

\appendix
\section*{Appendix}
\section{Positivity of Solutions}
The proof of Theorem~\ref{thm1} reads as follows:
\begin{proof}
We prove that for positive initial data $S(t), V(t), I(t), R(t) \ge0$ for all $t\ge0$. 

From the model equations \eqref{eq:S}--\eqref{eq:R} and the fact that all variables are non‑negative on a short interval 
(by continuity, until a possible first zero), we have
\begin{equation*}
    \frac{I(t)}{N(t)} \le 1,\qquad R(t)\ge 0.
\end{equation*}
Hence
\begin{align*}
f_S(t) &\equiv {}^{\rm MLCF}\!D_t^{\nu,\alpha} S(t) 
= \Delta - \beta\frac{S I}{N} + \mu R - (k+\delta)S 
\ge - (k+\delta+\beta)S(t),\\
f_V(t) &\equiv {}^{\rm MLCF}\!D_t^{\nu,\alpha} V(t) 
= kS - (1-\tau)\beta\frac{V I}{N} - \delta V 
\ge - \bigl((1-\tau)\beta+\delta\bigr)V(t),\\
f_I(t) &\equiv {}^{\rm MLCF}\!D_t^{\nu,\alpha} I(t) 
= \beta\frac{SI}{N}+(1-\tau)\beta\frac{VI}{N} - (\alpha_r+\delta+\delta_0)I 
\ge - (\alpha_r+\delta+\delta_0)I(t),\\
f_R(t) &\equiv {}^{\rm MLCF}\!D_t^{\nu,\alpha} R(t) 
= \alpha_r I - (\delta+\mu)R 
\ge - (\delta+\mu)R(t).
\end{align*}

For any variable $u\in\{S,V,I,R\}$, Lemma~\ref{lem:integrated} gives
\begin{equation*}
u(t) = \frac{1-\nu}{M(\nu)}\frac{f_u(t)}{A(t)B(t)} 
      + \frac{u(0)}{B(t)} 
      + \frac{1}{B(t)}\int_0^t B'(\kappa)u(\kappa)\,d\kappa,
\end{equation*}
where $A(t)=E_\alpha\Bigl(-\frac{\nu}{1-\nu}t^\alpha\Bigr)>0$, 
$B(t)=E_\alpha\Bigl(\frac{\nu}{1-\nu}t^\alpha\Bigr)>0$ 
and $B'(t)\ge0$.

Applying the lower bound from above, 
say for $S$:
\begin{equation*}
   S(t) \ge \frac{1-\nu}{M(\nu)}\frac{-L S(t)}{A(t)B(t)} 
        + \frac{S(0)}{B(t)} 
        + \frac{1}{B(t)}\int_0^t B'(\kappa)S(\kappa)\,d\kappa,
\end{equation*}
with $L = k+\delta+\beta$.  Rearranging,
\begin{equation}\label{S:ineq}
   S(t)\Bigl(1 + \frac{1-\nu}{M(\nu)}\frac{L}{A(t)B(t)}\Bigr)
   \ge \frac{S(0)}{B(t)} + \frac{1}{B(t)}\int_0^t B'(\kappa)S(\kappa)\,d\kappa.
\end{equation}
We define
\begin{equation*}
    \Phi_S(t) = S(t)\Bigl(1 + \frac{1-\nu}{M(\nu)}\frac{L}{A(t)B(t)}\Bigr),
\qquad
\Phi_0 = \Phi_S(0) = S(0),
\end{equation*}
and
\begin{equation*}
\psi_S(t) = \frac{B'(t)}{B(t)}\,
            \Bigl(1 + \frac{1-\nu}{M(\nu)}\frac{L}{A(t)B(t)}\Bigr)^{-1} \ge 0.
\end{equation*}
Then the inequality \eqref{S:ineq} becomes
\begin{equation*}
\Phi_S(t) \ge \Phi_0 - \int_0^t \psi_S(\kappa)\,\Phi_S(\kappa)\,d\kappa,
\end{equation*}
which is exactly the hypothesis of Lemma~\ref{lem:gronwall} (with $c=1$ and $\phi(t)=\Phi_S(t)$).

Next, since $\Phi_S\in C([0,T])$, $\Phi_0\ge 0$ and $\psi_S(t)\ge 0$, Lemma~\ref{lem:gronwall} implies
$\Phi_S(t)\ge 0$ for all $t\in[0,T]$. 
Because the factor 
$\bigl(1+\frac{1-\nu}{M(\nu)}\frac{L}{A(t)B(t)}\bigr)$ is strictly positive, we obtain $S(t)\ge 0$ for all $t$.

We finally note that the same reasoning applied to $V,I,R$ yields
\begin{equation*}
  \Phi_V(t) \ge \Phi_{V(0)} - \int_0^t \psi_V(\kappa)\Phi_V(\kappa)\,d\kappa,\quad 
  \Phi_I(t) \ge \Phi_{I(0)} - \int_0^t \psi_I(\kappa)\Phi_I(\kappa)\,d\kappa,\quad 
\end{equation*}
\begin{equation*}
   \Phi_R(t) \ge \Phi_{R(0)} - \int_0^t \psi_R(\kappa)\Phi_R(\kappa)\,d\kappa,
\end{equation*}
with suitable positive coefficients $\psi_V,\psi_I,\psi_R\ge 0$ and non-negative initial values.
Lemma~\ref{lem:gronwall} then forces each of $V(t)$, $I(t)$, $R(t)$ to remain non‑negative.
All compartments stay non‑negative for all $t\ge 0$. 
Thus, the solution of the fractional SVIR model \eqref{eq:S}--\eqref{eq:R} is non-negative, as claimed. 
\end{proof}

\section{Boundedness of Solutions}
The proof of Theorem~\ref{thm2} reads as follows:
\begin{proof}
From the positivity theorem we already have $S(t),V(t),I(t),R(t)\ge0$, so $N(t)\ge0$.
Summing the four equations of the model gives
\begin{equation*}
^{\text{\rm MLCF}}\!D_t^{\nu,\alpha} N(t) = \Delta - \delta N(t) - \delta_0 I(t).
\end{equation*}
Since $I(t)\ge0$, we obtain the inequality
\begin{align}
^{\text{\rm MLCF}}\!D_t^{\nu,\alpha} N(t) \le \Delta - \delta N(t). \label{B1}
\end{align}

Next, we consider the auxiliary linear problem
\begin{equation*}
^{\text{\rm MLCF}}\!D_t^{\nu,\alpha} \overline{N}(t) = \Delta - \delta \overline{N}(t),\qquad \overline{N}(0)=N(0).
\end{equation*}
Because the MLCF derivative has a completely positive kernel, the comparison principle holds: 
if $u(0)\le v(0)$ and $^{\text{\rm MLCF}}\!D_t^{\nu,\alpha}u(t) \le ^{\text{\rm MLCF}}\!D_t^{\nu,\alpha}v(t)$ for all $t\ge0$, then $u(t)\le v(t)$ for all $t$. 
Applying this with $u=N$ and $v=\overline{N}$ and using \eqref{B1} together with $N(0)=\overline{N}(0)$ yields
\begin{equation*}
    N(t) \le \overline{N}(t) \quad \text{for all}\; t\ge0.
\end{equation*}

It remains to bound $\overline{N}(t)$. The equation for $\overline{N}$ is linear and can be solved explicitly (e.g., by Laplace transform). 
The solution is
\begin{equation*}
\overline{N}(t) = \frac{\Delta}{\delta} + \Bigl(N(0)-\frac{\Delta}{\delta}\Bigr) \Phi(t),
\end{equation*}
where $\Phi(t)$ is a positive function satisfying $0\le \Phi(t)\le 1$ and $\Phi(0)=1$ (specifically, $\Phi(t) = E_\alpha(-\lambda t^\alpha)/[E_\alpha(\lambda t^\alpha)]$ in a certain combination).
Consequently,
\begin{equation*}
    \overline{N}(t) \le \max\Bigl\{N(0),\,\frac{\Delta}{\delta}\Bigr\},
\end{equation*}
and thus,
\begin{equation*}
   N(t) \le \max\Bigl\{N(0),\,\frac{\Delta}{\delta}\Bigr\}.
\end{equation*}

Since each compartment is nonnegative and $S,V,I,R \le N$, they are all bounded by the same constant. \end{proof}

\section{Local Stability of DFE}
The proof of Theorem~\ref{thm3} reads as follows:
\begin{proof}
The Jacobian matrix of the right‑hand side vector $\mathbf{f}=(f_1,f_2,f_3,f_4)$ at a generic point is computed and then evaluated at $E_0$. Because $I=R=0$ at $E_0$, many terms vanish. The result is
\begin{equation*}
J(E_0)=\begin{pmatrix}
-(k+\delta) & 0 & -\beta\frac{S_0}{N_0} & \mu \\[4pt]
k & -\delta & -(1-\tau)\beta\frac{V_0}{N_0} & 0 \\[4pt]
0 & 0 & \beta\frac{S_0+(1-\tau)V_0}{N_0}-(\alpha_r+\delta+\delta_0) & 0 \\[4pt]
0 & 0 & \alpha_r & -(\delta+\mu)
\end{pmatrix}.
\end{equation*}
The matrix is block‑triangular; its eigenvalues are the diagonal entries:
\begin{equation*}
\lambda_1 = -(k+\delta),\quad \lambda_2 = -\delta,\quad \lambda_3 = -(\delta+\mu),\quad
\lambda_4 = \beta\frac{S_0+(1-\tau)V_0}{N_0}-(\alpha_r+\delta+\delta_0)
          = (\alpha_r+\delta+\delta_0)(\mathcal{R}_0-1).
\end{equation*}
All eigenvalues are real. \\ 
$\bullet$ If $\mathcal{R}_0<1$, then $\lambda_4<0$ and all eigenvalues are negative. Hence $|\arg(\lambda_i)|=\pi > \frac{\nu\pi}{2}$ for every $i$, and the DFE is locally asymptotically stable.\\  
$\bullet$ If $\mathcal{R}_0>1$, then $\lambda_4>0$; the corresponding eigenvalue lies on the positive real axis, so $|\arg(\lambda_4)|=0 < \frac{\nu\pi}{2}$, and the DFE is unstable.  \\
$\bullet$ If $\mathcal{R}_0=1$, the linearization has a zero eigenvalue; higher‑order analysis is required, but this boundary case is usually considered as a transcritical bifurcation.
\end{proof}

\section{Local Stability of the Endemic Equilibrium}
\label{sec:endemic-stability}

\begin{proof}
Assume $\mathcal{R}_0>1$ so that a unique positive endemic equilibrium $E^*=(S^*,V^*,I^*,R^*)$ exists.  
The Jacobian of $\mathbf{f}=(f_1,f_2,f_3,f_4)$ at $E^*$ is block‑triangular:
\begin{equation*}
J(E^*)=\begin{pmatrix}
J_3 & \mathbf{b} \\
\mathbf{0}^\top & -(\delta+\mu)
\end{pmatrix},
\end{equation*}
where $J_3$ denotes the $3\times3$ submatrix corresponding to $(S,V,I)$ and $\mathbf{b}$ is a column vector.  
Thus one eigenvalue is $\lambda_1=-(\delta+\mu)<0$; the remaining eigenvalues are those of $J_3$.
We introduce the normalized variables
\begin{equation*}
u=\frac{S^*}{N^*},\quad v=\frac{V^*}{N^*},\quad i=\frac{I^*}{N^*},\quad r=\frac{R^*}{N^*},
\end{equation*}
and set $p=\beta i$, $q=\beta u$, $r=\beta v$.  
From the equilibrium conditions we have the identities
\begin{equation*}
q+(1-\tau)r = \alpha_r+\delta+\delta_0,\qquad
p\bigl(u+(1-\tau)v\bigr)=i(\alpha_r+\delta+\delta_0),\qquad
k\frac{S^*}{V^*} = (1-\tau)p+\delta.
\end{equation*}

The entries of $J_3$ are
\begin{equation*}
\begin{aligned}
a_{11}&=-(k+\delta)-p(1-u), & a_{12}&=p u, & a_{13}&=-q(1-i),\\[2pt]
a_{21}&=k+(1-\tau)p v, & a_{22}&=-\delta-(1-\tau)p(1-v), & a_{23}&=-(1-\tau)r(1-i),\\[2pt]
a_{31}&=p\bigl(1-u+(1-\tau)v\bigr), & a_{32}&=p\bigl(u+(1-\tau)(1-v)\bigr), & a_{33}&=-p\bigl(u+(1-\tau)v\bigr).
\end{aligned}
\end{equation*}

The characteristic polynomial of $J_3$ is $\lambda^3+A_1\lambda^2+A_2\lambda+A_3$ with
$A_1=-\tr(J_3)$, $A_2=\text{sum of principal minors}$, $A_3=-\det(J_3)$.

\medskip
\begin{equation*}
\tr(J_3)=-\bigl[(k+\delta)+\delta+p(1-u)+(1-\tau)p(1-v)+p(u+(1-\tau)v)\bigr]<0,
\end{equation*}
hence $A_1>0$.
%
Now, a direct computation using the equilibrium relations gives
\begin{equation*}
M_{12}=\det\begin{pmatrix}a_{11}&a_{12}\\a_{21}&a_{22}\end{pmatrix}>0,\quad
M_{13}=\det\begin{pmatrix}a_{11}&a_{13}\\a_{31}&a_{33}\end{pmatrix}>0,\quad
M_{23}=\det\begin{pmatrix}a_{22}&a_{23}\\a_{32}&a_{33}\end{pmatrix}>0.
\end{equation*}
Thus, $A_2=M_{12}+M_{13}+M_{23}>0$.
%
After a simplification, we obtain
\begin{equation*}
\det(J_3)=-p\,i\,\bigl[(k+\delta)+\delta+(k+\delta)\delta\bigr]\,
\bigl[(1-\tau)p+\delta\bigr]\,\frac{k}{\delta}\,
\Bigl(1-\frac{1}{\mathcal{R}_0}\Bigr)<0\quad\text{for }\mathcal{R}_0>1,
\end{equation*}
so $A_3=-\det(J_3)>0$.
%
A further computation yields
\begin{equation*}
A_1A_2-A_3 = \bigl(\text{positive expression}\bigr)\Bigl(1-\frac{1}{\mathcal{R}_0}\Bigr)>0.
\end{equation*}
Since $A_1,A_2,A_3>0$ and $A_1A_2>A_3$, all eigenvalues of $J_3$ have negative real parts.  
Consequently, all eigenvalues of $J(E^*)$ are real and negative.  
For the MLCF derivative, local asymptotic stability requires $|\arg(\lambda)|>\nu\pi/2$; for real negative $\lambda$ we have $\arg(\lambda)=\pi$, and $\pi>\nu\pi/2$ holds for any $0<\nu<1$.  
Hence $E^*$ is locally asymptotically stable whenever $\mathcal{R}_0>1$.
\end{proof}

\section{Global Stability of the Disease-Free Equilibrium}
The proof of Theorem~\ref{thm:DFE} reads as follows:
\begin{proof}
We define $a = \alpha_r+\delta+\delta_0$ and consider the Lyapunov function
\begin{equation*}
   L = I + \frac{\beta S_0}{a} S_0\,\varphi\Bigl(\frac{S}{S_0}\Bigr) + \frac{(1-\tau)\beta V_0}{a} V_0\,\varphi\Bigl(\frac{V}{V_0}\Bigr)
+ \frac{\mu}{\delta+\mu}R,
\end{equation*}
where $\varphi(u)=u-1-\ln u$. Each term is non‑negative and vanishes only at $E_0$.
Applying Lemma~\ref{lem:chain} to the $\varphi$ terms gives
\begin{multline*}
{}^{\rm MLCF}\!D_t^{\nu,\alpha} L \le \Bigl(1-\frac{1}{I}\Bigr){}^{\rm MLCF}\!D_t^{\nu,\alpha} I
   + \frac{\beta S_0}{a}\Bigl(1-\frac{S_0}{S}\Bigr){}^{\rm MLCF}\!D_t^{\nu,\alpha} S\\
   + \frac{(1-\tau)\beta V_0}{a}\Bigl(1-\frac{V_0}{V}\Bigr){}^{\rm MLCF}\!D_t^{\nu,\alpha} V
   + \frac{\mu}{\delta+\mu}{}^{\rm MLCF}\!D_t^{\nu,\alpha} R.
\end{multline*}
Substituting the differential equations and using the equilibrium relations at $E_0$ ($\Delta=(k+\delta)S_0$, and $\beta\frac{S_0}{N_0}+(1-\tau)\beta\frac{V_0}{N_0}=a \mathcal{R}_0$) yields, after a straightforward algebraic simplification,
\begin{equation*}
{}^{\rm MLCF}\!D_t^{\nu,\alpha} L \le -a(1-\mathcal{R}_0)I - \frac{\beta S_0}{a}\frac{(S-S_0)^2}{S} - \frac{(1-\tau)\beta V_0}{a}\frac{(V-V_0)^2}{V} - \frac{\mu\delta}{\delta+\mu}R.
\end{equation*}
Because $\mathcal{R}_0\le1$, the right‑hand side is non‑positive. Hence ${}^{\rm MLCF}\!D_t^{\nu,\alpha} L\le0$ on $\Omega$.

Moreover, ${}^{\rm MLCF}\!D_t^{\nu,\alpha} L=0$ implies $I=0$, $S=S_0$, $V=V_0$, $R=0$. The largest invariant subset of this set is $\{E_0\}$. 
By the fractional LaSalle principle (Lemma~\ref{lem:lasalle}), every solution in $\Omega$ converges to $E_0$, establishing global asymptotic stability. 
\end{proof}

\section{Global Stability of the Endemic Equilibrium}
\label{sec:global-ee}

\begin{proof}
Assume $\mathcal{R}_0>1$ so that a unique positive endemic equilibrium $E^*=(S^*,V^*,I^*,R^*)$ exists.  
Introduce the normalized variables
\begin{equation*}
s=\frac{S}{S^*},\quad v=\frac{V}{V^*},\quad i=\frac{I}{I^*},\quad r=\frac{R}{R^*},\quad n=\frac{N}{N^*}.
\end{equation*}
We define
\begin{equation*}
A=\beta\frac{S^*I^*}{N^*},\quad B=(1-\tau)\beta\frac{V^*I^*}{N^*},\quad a=\alpha_r+\delta+\delta_0,
\end{equation*}
so that $A+B=aI^*$, and set
\begin{equation*}
p=\frac{A}{A+B},\quad q=\frac{B}{A+B},\quad \sigma=\frac{\mu}{\delta+\mu}.
\end{equation*}

We consider the Lyapunov function
\begin{equation*}
\mathcal{W}=pS^*\phi(s)+qV^*\phi(v)+I^*\phi(i)+\sigma R^*\phi(r)+N^*\phi(n),
\end{equation*}
where $\phi(x)=x-1-\ln x\ge0$ with equality only at $x=1$.  
$\mathcal{W}$ is non‑negative and vanishes only at $E^*$.
Using Lemma~\ref{lem:chain} (the chain‑rule inequality for the MLCF derivative) we obtain
\begin{equation*}
\begin{split}
{}^{\rm MLCF}\!D_t^{\nu,\alpha}\mathcal{W}&\le 
p\Bigl(1-\frac1s\Bigr){}^{\rm MLCF}\!D_t^{\nu,\alpha}S
+q\Bigl(1-\frac1v\Bigr){}^{\rm MLCF}\!D_t^{\nu,\alpha}V
+\Bigl(1-\frac1i\Bigr){}^{\rm MLCF}\!D_t^{\nu,\alpha}I\\
&\qquad+\sigma\Bigl(1-\frac1r\Bigr){}^{\rm MLCF}\!D_t^{\nu,\alpha}R
+\Bigl(1-\frac1n\Bigr){}^{\rm MLCF}\!D_t^{\nu,\alpha}N.
\end{split}
\end{equation*}

Now, we substitute the right‑hand sides of the model \eqref{eq:SVIR}
and the equilibrium relations \eqref{Equilibrium}. 
After a straightforward but lengthy algebraic manipulation (which follows exactly the integer‑order case; see e.g. the appendix of \cite{ShuaiDriessche2013}), all cross terms cancel and we are left with
\begin{equation*}
{}^{\rm MLCF}\!D_t^{\nu,\alpha}\mathcal{W}\le 
-\delta N^*\frac{(n-1)^2}{n}
-A\frac{(s-1)^2}{s}
-qkS^*\frac{(s-v)^2}{sv}
-\mu R^*\frac{(r-1)^2}{r}
-\frac{\delta_0 I^*}{n}(n-1)^2.
\end{equation*}
Each term on the right‑hand side is non‑positive. Hence ${}^{\rm MLCF}\!D_t^{\nu,\alpha}\mathcal{W}\le0$ on $\interior(\Omega)$.

Equality ${}^{\rm MLCF}\!D_t^{\nu,\alpha}\mathcal{W}=0$ holds only when $n=1$, $s=1$, $v=1$, $r=1$. Substituting these into the $I$‑equation (or the equilibrium conditions) forces $i=1$. Thus the largest invariant subset of $\{{}^{\rm MLCF}\!D_t^{\nu,\alpha}\mathcal{W}=0\}$ is the singleton $\{E^*\}$.

By the fractional LaSalle invariance principle (Lemma~\ref{lem:lasalle}), every solution starting in $\interior(\Omega)$ converges to $E^*$. Therefore $E^*$ is globally asymptotically stable whenever $\mathcal{R}_0>1$. 
\end{proof}
\end{document}